\documentclass[oneside,12pt]{amsart}
\usepackage[fontsize=14pt]{scrextend}
\usepackage[scale={.85,.85}]{geometry}
\usepackage{amsmath}
\usepackage{mathrsfs}
\usepackage{mathptmx,amssymb}
\usepackage[dvipsnames]{xcolor}
\definecolor{BLUE}{RGB}{41,86,143}
\definecolor{RED}{RGB}{178,31,53}
\usepackage{graphicx}
\usepackage{subcaption}
\usepackage{multirow}

\usepackage[bbgreekl]{mathbbol}
\usepackage{bbold}
\usepackage{bm}
\usepackage{accents}
\DeclareSymbolFontAlphabet{\amsmathbb}{AMSb}

\usepackage{multicol}

\newtheorem{teo}{Theorem}
\newtheorem{lemma}{Lemma}
\newtheorem{coro}{Corollary}
\newtheorem{propo}{Proposition}

\theoremstyle{definition}
\newtheorem*{defi}{Definition}
\newtheorem{remark}{Remark}
\newtheorem{example}{Example}

\newcommand{\bb}[1]{\mathbb{#1}}
\newcommand{\E}{\ensuremath{ \bb{E} } }

\newcommand{\bo}[1]{\ensuremath{{\bf #1 } }}
\title[Fast Simulation of Conditioned MBGW]{\Large Fast Simulation of Size-Constrained Multitype Bienaym\'e-Galton-Watson Forests and Applications}
\author{Osvaldo Angtuncio Hern\'andez}
\address[OAH]{Centro de Investigaci\'on en Matem\'aticas\\ Guanajuato,  
M\'exico}
\thanks{Research supported partially by CoNaCyT grant FC-2016-1946,  UNAM-DGAPA-PAPIIT grant no. IN114720 and the Deutsche Forschungsgemeinschaft (through grant DFG-SPP-2265)} 
\email{osvaldo.angtuncio@cimat.mx}
\usepackage[colorlinks,citecolor=Plum,urlcolor=Periwinkle,linkcolor=DarkOrchid,backref=page]{hyperref}

\newcommand{\re}{\ensuremath{\mathbb{R}}}
\newcommand{\paren}[1]{\ensuremath{\left( #1\right) }}

\newcommand{\p}{\mathbb{P}}
\newenvironment{esn}{\begin{equation*}}{\end{equation*}}

\newcommand{\se}{\ensuremath{\mathbb{E}}}

\newcommand{\bra}[1]{\ensuremath{\left[ #1\right] }}
\newcommand{\z}{\ensuremath{\mathbb{Z}}}
\newcommand{\na}{\ensuremath{\mathbb{N}}}

\newcommand{\indi}[1]{\si\left\{ #1\right\}}
\newcommand{\si}{{\ensuremath{\bf{1}}}}

\newcommand{\esp}[1]{\ensuremath{\se\! \left( #1 \right)}}

\newcommand{\proba}[1]{\ensuremath{\sip\! \left( #1 \right)}}
\newcommand{\sip}{\mathbb{P}}

\newcommand{\floor}[1]{\ensuremath{\lfloor #1\rfloor}}

\usepackage{algorithm}
\makeatletter
\def\BState{\State\hskip-\ALG@thistlm}
\makeatother
\usepackage{algcompatible}
\begin{document}
\maketitle

\vspace{-.5cm}

\vspace{.1cm}
\begin{center}
		\emph{Centro de Investigaci\'on en Matem\'aticas,}\\ \emph{Guanajuato, M\'exico}\\
\end{center}

\vspace{.3cm}

\begin{abstract}
The degree sequence $(n_{i,j}(k), 1\leq i,j\leq d, k\geq 0)$ of a multitype forest with $d$ types encodes the number of individuals of type $i$ with $k$ children of type $j$.
In this paper, we introduce a simple algorithm to sample a multitype forest uniformly from the set of all forests with a given degree sequence (MFGDS).
This generalizes the single-type construction of Broutin and Marckert (2014).
To achieve this, we extend the Vervaat transform (1979) to multidimensional discrete exchangeable increment processes.

We demonstrate that MFGDS extend multitype Bienaym\'e--Galton--Watson (MBGW) forests.
Specifically, mixing MFGDS laws recovers MBGW forests conditioned on a fixed size for each type (CMBGW).
Under general assumptions, we derive the law of the total population by types in an MBGW forest and relate it to a multidimensional first-hitting time.
This result, which is of independent interest, generalizes the Otter--Dwass (1949,1969) and Kemperman (1950) formulas.

By combining this relation with our MFGDS construction, we provide an efficient algorithm to simulate CMBGW forests, generalizing the work of Devroye (2012).
When the variance is finite, the expected simulation time outperforms standard na\"ive methods.
For the proof we derive a generalized local limit theorem for multidimensional first-hitting times.
Finally, we apply our results to enumerate plane, labeled, and binary multitype forests with fixed sizes, generalizing results of Pitman (1998).
\end{abstract}

\tableofcontents

\section{Introduction}

\emph{Bienaym\'e-Galton-Watson forests} (BGW forests) are a simplified model for the genealogy of populations, where individuals have the same reproduction law. 
A natural generalization of such a model are the \emph{multitype Bienaym\'e-Galton-Watson forests} (MBGW forests), used when several types of individuals coexist (leading to different reproduction laws). 
Such MBGW forests have applications in biology, demography, 
genetics, medicine,
epidemics, 
language theory, and several other sciences
(see \cite{MR0163361,MR0286196,MR0279901,MR0488341,MR0392022,Ciampi86,MR1601681,MR2560499,articleBranchingProcessesTheirRoleInEpidemiology,MR3310028,MR3941869,MR3992183}). 
Those models have applications also for pure mathematics.
Roughly speaking, Miermont \cite{MR2469338} has proved that under certain conditions (in particular, a finite variance condition), when \emph{ignoring the types}, large MBGW forests are similar to large (unitype) BGW forests. 
The latter result was extended by \cite{MR3606739} to an infinite number of types, and by \cite{MR3748121} to forests with offspring distribution in the domain of attraction of a stable law.

\emph{Conditioned random forests} also provide us with applications in different branches of science. 
In biology, they can be useful to model mutations in mitochondrial DNA  \cite{MR1934415}. 
An application to epidemiological risk analysis is given in \cite{Penisson2010}, and in \cite{MR3363681}, in the multitype case.
In probability, unitype conditioned random forests can be used to describe \emph{Real Trees}, to enumerate \emph{Classes of Combinatorial Forests}, and to study \emph{Random Graphs}. 
Aldous' theory of the  \emph{Continuum Random Tree} (CRT),  started proving that the CRT is the scaling limit of (unitype) finite variance BGW trees \emph{conditioned on their total population size} \cite{MR1085326,MR1166406}. Together with the works of Le Gall and Le Jan \cite{MR1617047}, and Duquesne and Le Gall \cite{MR1954248}, those  authors opened the path to the study of random \emph{real trees}.
Conditioned random forests also
coincide with some combinatorial models, providing us with \emph{enumerations of combinatorial classes of forests}, for example, Cayley, plane, $k$-ary or binary forest (see \cite{MR1630413}).
Not only conditioning with the size of the random forest, but also with the number of leaves, the height, the width, the number of individuals having fixed children, or the time of extinction, has been proved useful (see \cite{MR1166406,MR1630413,MR1678582,MR2888318,MR2908619,MR2946438,MR3188597,2020arXiv200812242A}).
Another important application of conditioned BGW trees, is their appearance in the study of local windows of the \emph{Erd\H{o}s-R\'enyi random graph} \cite{MR0125031}, which allows us to approximate  characteristics like the size of the \emph{giant component} in the random graph (see Chapter 4.2 in \cite{MR3617364}).
A model that goes beyond the Erd\H{o}s-R\'enyi random graph, is the \emph{configuration model}, introduced by \cite{MR505796,MR595929}. 
The interest in such random graphs lies on their matching with observations of large \emph{real-world networks}.
With no doubt, studying 
\emph{random trees with a prescribed degree sequence}  helps to understand invariance principles of the configuration model (see also the discussion in \cite{MR3188597}).

The literature extending the previous paragraph to conditioned multitype random forest is less diverse. 
When proving their \emph{local} convergence towards a limiting object, such as the multitype generalization of Kesten's infinite tree, some results are known \cite{MR653216,MR3803914,MR3769811}.
But for the \emph{global} convergence this is not the case.
In the same work of Miermont \cite{MR2469338}, the author proves (again ignoring the types) that MBGW forests conditioned to have a large number of individuals of a \emph{fixed} type, behave as the CRT. 
The latter was generalized to the stable case in \cite{MR3748121}. 
Some other conditionings in the literature are that a linear combination of the sizes by types is fixed, that the size by types approaches a certain eigenvector, or that the forests goes extinct in a large time (see \cite{MR2383922,MR2870524,MR3476213,MR3803914,MR3769811,MR3963289,2019arXiv191207296H}). 
Naturally, one is tempted to expect a convergence of CMBGW forests to a \emph{Continuum Multitype Random forest}. 
The author and David Clancy \cite{HernandezClancy2025multitypelevytrees} have recently shown that a large class of MBGW trees conditioned in some sense to be large, converge to a multitype L\'evy tree.
The difference with other results, is that the limiting object retains a multitype structure.

In this paper, we work with MBGW forests with $r_i\geq 0$ roots type $i$, conditioned to have a total number $n_i\geq r_i$ of individuals type $i$, for $i\in[d]:=\{1,\ldots, d\}$, that is, conditioned with the total number of individuals by types (CMBGW($\bo n,\bo r$) forests). 
Our main result in this paper is to generalize the construction in \cite{MR2888318}, where Devroye simulates a (unitype) BGW tree conditioned to have size $n$. 
We provide an exact sampling algorithm that produces CMBGW($\bo n,\bo r$) forests. 
Our algorithm is simple to implement, and under mild conditions, it performs faster than an usual na\"ive method. 

Devroye's results relies on the Otter-Dwass formula \cite{MR0030716,MR0253433}.
The latter says that the law of the total number of individuals in a BGW forest $\tau_k$ with $k$ trees, having offspring distribution $\mu$ is given by
\begin{equation*}
	\proba{\# \tau_k=n}=\frac{k}{n}\proba{X_n=n-k}\qquad n\geq k,
\end{equation*}
where $X$ is a random walk with step law $\mu$. 
It is also well-known that the size of a forest has the same law as the first hitting time to 0 of a random walk started at the number of trees, which is Kemperman's formula (see \cite{MR0038045,MR0138139}). 
Understanding the total progeny of a forest is interesting on its own (see \cite{MR1176953,MR1628630}, for a multitype model where the total number of individuals is studied).
One reason for this, is that it gives us a nice connection between probability and combinatorics. 
Using the law of $\# \tau_k$, one can deduce the total number of plane, labeled and binary forests having $k$ trees and $n$ vertices, see \cite{MR1630413}. 
An example in the unitype case is the following. 
Let $\mathcal{F}^{plane}_{n,k}$ be a uniformly distributed forest from the set of plane forests having $k$ trees and $n$ individuals;
let $\mathcal{G}_{k,p}$ be a BGW forest with $k$ trees and Geometric($p$) offspring distribution, with $p\in (0,1)$. 
Then 
\begin{equation*}
	\mathcal{F}^{plane}_{n,k}\stackrel{d}{=}\paren{\mathcal{G}_{k,p}\,|\,\#\mathcal{G}_{k,p}=n }\ \ \ \ \ \ \mbox{and}\ \ \ \ \ \ \proba{\mathcal{G}_{k,p}=\bo f\,|\,\#\mathcal{G}_{k,p}=n }=\frac{1}{\frac{k}{n}{2n-k-1\choose n-k}},
\end{equation*}for every plane forest $\bo f$ with $k$ trees and $n$ individuals. 
Similar equalities in distribution are available for the Poisson and the Bernoulli distribution.

Our second main result is to generalize the above Otter-Dwass formula (and also Kemperman's formula) and those connections between combinatorics and probability about enumerations of forests and lattice paths, to the multitype setting.
The arguments employed in our proof are sufficiently robust to extend to the continuous-state setting. 
Indeed, an application of our generalization of the Otter-Dwass formula, is given in a separate paper \cite{AngtuncioPalauConvergenceMBGWProcess2026}. 
In such a work, together with Sandra Palau, we prove the convergence of the CMBGW process towards a continuum multitype limit, which can be thought as the \emph{multitype continuous state branching process} (MCSBP) \emph{conditioned to have fixed mass by types}.
The latter is a generalization to the multitype setting, of Conjecture 4 of Aldous \cite{MR1166406}, and of the convergence of the branching process given in \cite{MR2534486}. 
Note that an (unconditioned) MCSBP can be constructed using the results in \cite{MR3689968}. 
Also, a clear relation between the sizes of each type in the components of multitype random graphs (a.k.a. inhomogeneous random graphs) and MBGW trees is still missing (some connections can be found in Chapter 3.4-3.5 of \cite{Hofstad2022}, and Lemma 4.7 in \cite{MR4586227}). 
We believe our construction can be useful to understand the latter.

Using the generalization of the Otter-Dwass formula, we give an algorithm to simulate CMBGW forests.
The latter generalizes  Devroye's algorithm \cite{MR2888318} for single-type trees. 
Let us state Devroye's result, so the reader can understand the main steps to generalize it.
Instead of working directly with forests, we use a well-know bijection with \emph{excursions}. 
Given a unitype plane tree $\bo t$, label the vertices in a \emph{breadth-first order} (BFO), that is, starting with the root having label one and being at generation zero, label its children (constituting generation one) from left to right, label its grandchildren (constituting generation two) from left to right, and so on. 
We define the 
\emph{breadth-first walk} (BFW) as a path $W$ starting at one, and having increments $\kappa_{\bo t}(i)-1$, where $\kappa_{\bo t}(i)$ is the number of children of individual $i$ in breadth-first order. 
It is well-known that this coding of trees using the BFW (also called an \emph{excursion}) is indeed a bijection \cite{MR2245368}[Lemma 6.3, Chapter 6]. 

Devroye's algorithm to simulate a BGW tree of size $n$ and offspring distribution $\mu=(\mu_i,i\geq 0)$, and prove its efficiency, can be decomposed in five main steps:

\begin{enumerate}
	\item generate a multinomial vector $\bo{S}=(N_0,N_1,\ldots)$ with parameters $(n;\mu_0,\mu_1,\ldots)$, 
	\item repeat until $1+\sum iN_i=n$, that is, until $\bo{S}$ is a \emph{degree sequence} of some tree having size $n$ (here $N_i$ represents the number of individuals in a tree having $i$ children),
	\item define the \emph{child sequence} $\bo{c}(\bo{S}):=\bo{c}=(c_1,\ldots, c_n)$, which is a vector with $N_0$ zeros, $N_1$ ones, and so on, construct the walk $W^b$ having increments $(c_{\pi(j)}-1,j\in [n])$, where $\pi$ is a uniform random permutation on $[n]$, 
	and let $W$ be the \emph{Vervaat transform} of $W$, a walk with increments $(c_{\pi(i^*+j)}-1,j\in [n])$, where $i^*+j$ is considered modulo $n$ and $i^*$ is the first time $W^b$ reaches its minimum value, 
	\item the excursion $W$ has the same law as the BFW of a BGW tree conditioned to have size $n$, 
\item show that this algorithm takes an expected time of smaller order than a naive method. 
\end{enumerate}

Our simple algorithm (which we state now informally), to simulate a MBGW forest conditioned on having $\bo{n}=(n_1,\ldots, n_d)$ individuals of each type, and with offspring distribution $\bm \mu =(\mu_1,\ldots, \mu_d)$, where each $\mu_i$ is a distribution on $\mathbb{Z}^d_+$, requires to generalize the previous steps: 

\begin{enumerate}
	\item define the vector of vectors $\bb{S}=(\bo{S}_{
		i,j},i,j\in [d])$ where $\bo{S}_{i,j}=(N_{i,j}(k),k\in \z_+)$ has a multinomial distribution with parameters $(n_i;\mu_{i,j}(0),\mu_{i,j}(1),\ldots)$,
	\item repeat until it is a multitype degree sequence,
	\item generate a uniform multitype forest with given degree sequence $\bb{S}$ (UMFGDS), that is, construct a walk $\bb W^b$ using the uniformly permuted multitype child sequence of $\bb{S}$ and apply to this a multidimensional Vervaat transform, obtaining a multidimensional BFW,
	\item obtain the joint law of the total size by types of a MBGW forest, 
	\item join the previous results to construct a forest from the degree sequence $\bb S$, which has a law proportional to the law of a CMBGW forest,
\item show that this algorithm takes an expected time of smaller order than a naive method. 
\end{enumerate}
All the steps (1)--(6) constitute our novel theoretical and algorithmic contributions. 
More importantly, except for Step (6), our hypotheses are quite general, requiring the standard assumptions for the forest to go extinct (non-simplicity, recurrence, (sub)criticality); and the extra assumption of independence between types.  
To show Step (6), we assume a finite variance condition and a condition to simplify our generalized Otter-Dwass formula and compute asymptotic terms. 
Nevertheless, we also prove a lemma showing that, a finite variance condition and an offspring distribution independent of the parent's type are enough for Step (6). 
To prove that our algorithm outperforms a naive method, we also show a local limit theorem-like result for the joint law of the sizes by types of the forest, a result interesting on its own.

To establish these steps, we must first introduce the multidimensional breadth-first walk encoding of multitype forests (relying on \cite{MR3449255}) and characterize the \emph{good cyclical permutations} that identify when a multidimensional path successfully codes a valid forest. This allows us to define multitype degree sequences and generalize the Vervaat transform to multiple dimensions to map 'multidimensional bridges' into valid forest-coding paths. This leads to our first main result, Theorem \ref{teoConstructionOFUniformMtypeForestWithGDS}, which provides an exact procedure to construct a uniform multitype forest with a given degree sequence, fulfilling Step (3).
This extends to the multidimensional setting, Lemma 7 in \cite{MR3188597}. 

The theoretical importance of UMFGDS lies in their intrinsic connection to CMBGW forests. We demonstrate that under an independence assumption on the offspring distribution, CMBGW forests are explicitly finite mixtures of UMFGDS (cf.\ Proposition 8 in \cite{randomTreesHaveHeightO_nLouigi2022}). To evaluate the precise probabilities required for Step (4), we generalize the classical Otter--Dwass formula to the multitype setting. Theorem \ref{teoLawOfTheTotalPopulationByTypes} provides the exact joint law of the number of individuals of each type in a MBGW forest, from which we can directly deduce the exact law of the total population (Corollary \ref{coroLawOfTheTotalPopulationInAMBGW}).
Both of these results are interesting on their own. 

Integrating these elements culminates in Theorem \ref{teoSimulationOfCMBGWIntro}, our main algorithmic contribution, which provides a fast and exact rejection-based method to simulate CMBGW forests (Step 5). We further prove in Theorem \ref{teoIntroComplexityTime} that our algorithm rigorously outperforms a standard na\"ive algorithm in terms of expected complexity. 

At the end of the main results section, we apply our probabilistic framework of Theorem \ref{teoLawOfTheTotalPopulationByTypes}, to derive explicit enumeration formulas for uniform multitype plane, labeled, and binary forests with prescribed roots and sizes by types.

\textit{Organization of the paper:}
In Section \ref{sectionMainResults} we present our main results, providing the complete framework to state them formally.
In Section \ref{sectionConstructionOfMFGDS} we prove Theorem \ref{teoConstructionOFUniformMtypeForestWithGDS},  which constructs UMFGDS as the multidimensional Vervaat transform of a bridge obtained from a multidimensional degree sequence. 
Section \ref{sectionLawOfNumberOfIndividualsByTypes} is devoted to proving the joint law of the number of individuals by types (Theorem \ref{teoLawOfTheTotalPopulationByTypes}) and the total population law (Corollary \ref{coroLawOfTheTotalPopulationInAMBGW}). Furthermore, we supply explicit examples satisfying our hypotheses to derive the enumerations of specific combinatorial forests in Subsection \ref{subsectionExamplesMultitypeGeoPoiBer}. In Section \ref{sectionAlgorithmsConditionedRandomForests}, we establish that CMBGW forests are mixtures of UMFGDS and present the proof of our main algorithm, Theorem \ref{teoSimulationOfCMBGWIntro}. Section \ref{sectionExpectedComplexity} provides a complexity analysis of our approach compared to a na\"ive method.

\section{Main results}\label{sectionMainResults}

\subsection{Coding multitype forests}\label{subsubsectionCodingForests}

A \emph{rooted plane tree} $\bo t$ is a connected graph with no cycles having a distinguished vertex, together with a natural identification of each vertex by a finite sequence of non-negative integers (denoting its location on the tree). 
A \emph{rooted plane forest} is a planar graph whose connected components are rooted plane trees.
In this paper, we only consider forests consisting of a finite number of finite trees. 
For simplicity we just refer to them as trees or forests. 

We consider forests where each tree is labeled according to the BFO. 
Thus, the individuals of a plane forest $\bo f$ can be labeled as $(u_i;i\in \#\bo f)$ in BFO, where $\#\bo f$ is the \emph{total number of vertices} of $\bo f$. 
The \emph{degree sequence}  ${\bf s}=(n_0,n_1,\ldots)$ of a forest $\bo f$ satisfies $n_i:=n_i(\bo f)=|\{u\in \bo f:\kappa_{\bo f}(u)=i \}|$, where $\kappa_{\bo f}(u:=)\kappa(u)$ is the number of children of individual $u$ in the forest. 

Now, we briefly recall the BFW coding of a forest in the multitype case, following  \cite{MR3449255}.
Define $[n]=\{1,\ldots,n \}$ and $[n]_0=\{0,1,\ldots,n \}$ for $n\in\na$. 
For a forest $\bo f$, let $\rm{type}_{\bo f}:\bo f\mapsto [d]$ be an application from the set of vertices of $\bo f$ to $[d]$, such that the children of each vertex are ordered by its type, that is, if $u_i,u_{i+1}\ldots, u_{i+j}\in \bo f$ have the same parent, then $\rm{type}_{\bo f}(u_i)\leq \rm{type}_{\bo f}(u_{i+1})\leq \cdots \leq \rm{type}_{\bo f}(u_{i+j})$. 
The couple $(\bo f,\rm{type}_{\bo f})$ is called a \emph{$d$-multitype forest}, and to ease notation we only denote by $\bo f$ the multitype forest. 
If the forest $\bo f$ has $r_i$ trees with roots type $i$ for all $i$, we order the trees according to the type, so that $\rm{type}_{\bo f}(u_i)\leq \rm{type}_{\bo f}(u_{j})$ whenever $u_i$ and $u_j$ are roots. 
A \emph{subtree of type $i$} of $\bo f$, denoted by $\bo t^{(i)}$, is a maximal connected subgraph of $\bo f$ whose vertices are all of type $i$. 
Subtrees of type $i$ are ranked according to the order of their roots, and with this ordering, we define the \emph{subforest of type $i$} of $\bo f$ as $\bo f^{(i)}=\{\bo t^{(i)}_1, \ldots, \bo t^{(i)}_k,\ldots\}$.
We denote by $(u_n^{(i)};n\geq 1)$ the labeling of the subforest $\bo f^{(i)}$ in its own BFO. 
For $u\in \bo f$, denote by $\kappa_i(u)$ the number of children of type $i$ of $u$.  
We denote by $\# \bo f:=(\#_1 \bo f,\ldots, \#_d \bo f)$, the vector with $i$-th entry the number of vertices in the subforest $\bo f^{(i)}$ of $\bo f$. 
For a forest $\bo f$ with $\#_i\bo f=n_i$ for every $i$, the coding of the forest, called also the  \emph{breadth-first walk} (BFW) of $\bo f$, are the $d$-dimensional chains $\bo x^{(i)}=(x^{i,1},\ldots, x^{i,d})\in \mathbb{Z}^d$ with length $n_i$, defined for $0\leq n\leq n_i-1$ if $n_i\geq 1$, by
\begin{equation}\label{eqnChainsEncodingAMultitypeTree}
	x^{i,j}_{n+1}-x^{i,j}_{n}=\kappa_j(u_{n+1}^{(i)}) -\indi{i=j}
	\ \ \ \ \ \ \ \ \ \ i,\,j\,\in\,[d],
\end{equation}with $\bo x_0^{(i)}=\bo 0$. 
In Figure \ref{figBFOandSUbtrees} we show the BFO and BFW of a forest. 

\begin{figure}
	\begin{tabular}{ccc}
		\multirow{4}{*}[1.8cm]{\includegraphics[clip, trim=4.6cm 16.5cm 4.8cm 4.3cm, width=.45\textwidth]{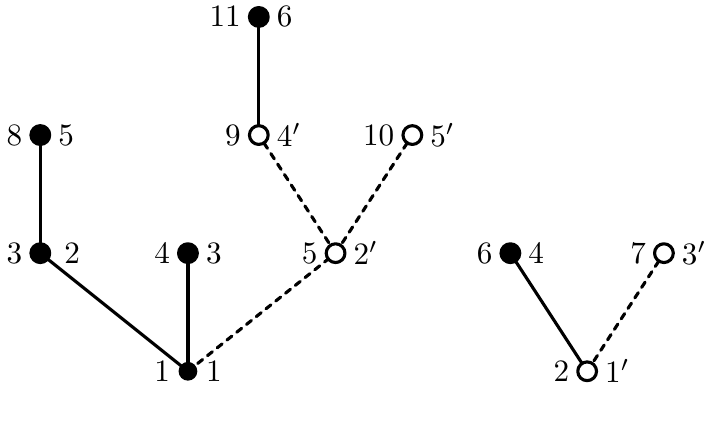} }&
		\vspace{.4cm} \includegraphics[clip, trim=4.5cm 18cm 10cm 4.5cm, width=.25\textwidth]{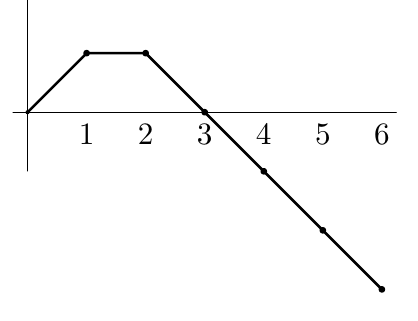} &  \vspace{-1.4cm}
		\includegraphics[clip, trim=4.5cm 18cm 10cm 4.5cm, width=.25\textwidth]{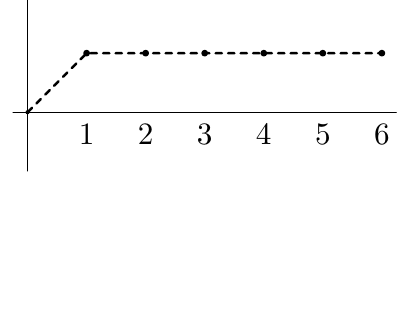} \\
		& $x^{1,1}$ & $x^{1,2}$ \\[6pt]
		& \vspace{-1cm} \includegraphics[clip, trim=4.5cm 19cm 11cm 4.5cm, width=.25\textwidth]{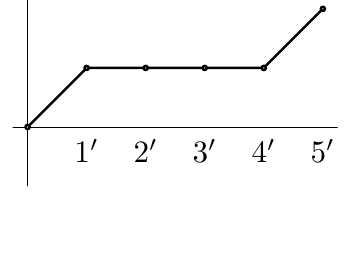} &   
		\includegraphics[clip, trim=4.5cm 19cm 11cm 4.5cm, width=.25\textwidth]{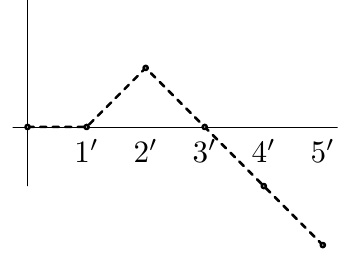} \\
		& $x^{2,1}$ & $x^{2,2}$ \\[6pt]
	\end{tabular}
	\caption{A multitype forest $\bo f$ with 2 types. Type 1 individuals are full circles, while type 2 individuals are empty circles. When a vertex has a type 1 child the edge is depicted continuous, whereas for a type 2 child is depicted dashed. The left-hand side label of the vertices is the breadth-first order of the (multitype) forest, and the right-hand side label is the BFO of the type 1 and type 2 subforest, labeled by $1,2,3,4,5,6$ and $1',2',3',4', 5'$ respectively. The plot on the right-hand side is the BFW $(x^{i,j})$ of $\bo f$.}\label{figBFOandSUbtrees}
\end{figure}

To describe the set of all possible chains $(\bo x^{(i)},i\in[d])$ that code a multitype forest, we require some definitions.
For $n\in \na$, 
consider any application $\bo y: [n]_0\mapsto \mathbb{Z}^d$ with $\bo y=(y(0),\ldots, y(n))$ and $ y(0)=0$. 
The \emph{$n$-cyclical permutations} of $\bo y$ are the $n$ applications $\theta_{q,n}(\bo y)$, for $q\in [n-1]_0$ given by
\[ \theta_{q,n}(\bo y)=\begin{cases} 
	y(j+q)-y(q) & j\leq n-q\\
	y(j+q-n)+y(n)-y(q) & n-q\leq j\leq n.
\end{cases}
\]We say that the path $\bo y:\na\mapsto \z$ is a \emph{downward skip-free chain}, if $y_{k+1}-y_k\in \z_+\cup\{-1 \}$.
All cyclical permutations of the BFW of multitype forests can be described as follows.

\begin{defi}\label{defiCodifyingSequencesOfMultitypeForests}
	Fix any $\bo{n}=(n_1,\ldots, n_d)\in \z^d_+$, and define $S_d$ as the set of $[\mathbb{Z}^d]^d$-valued sequences $\bb x=(\bo x^{(1)},\ldots, \bo x^{(d)})$ such that for all $i\in [d]$, $\bo x^{(i)}=(x^{i,1},\ldots, x^{i,d})$ is a $\mathbb{Z}^d$-valued sequence starting at zero with domain $[n_i]_0$, and where $x^{i,j}=(x^{i,j}_k,k\in [n_i]_0)$ is non-decreasing when $i\neq j$, and a downward skip-free chain when $i=j$. 
\end{defi}Sequences $\bb x\in S_d$ will be denoted also by $\bb x=(x^{i,j}_k,k\in [n_i]_0,i,j\in [d])$, and the vector $\bo{n}=(n_1,\ldots, n_d)\in \mathbb{Z}^d_+$, is called the \emph{length} of $\bb x$.
The 
\emph{${\bf n}$-cyclical permutations} of $\bb x\in S_d$ are given by
\begin{equation*}
	\theta_{{\bf q},{\bf n}}(\bb x):=(\theta_{q_1,n_1}(\bo x^{(1)}),\ldots, \theta_{q_d,n_d}(\bo x^{(d)}))\ \ \ \ \ \ \ \ \ \forall\ {\bf q}=(q_1,\ldots, q_d) \mbox{ such that }0\leq {\bf q}\leq {\bf n}-\bo 1_d,
\end{equation*}with $\bo{1}_d=(1,\ldots, 1)$ a vector of length $d$, and where we write $\bo{m}\leq\bo{n}$ if the inequality holds component-wise, for $\bo{m,n}\in \z^d_+$. We write $\bo{m}<\bo{n}$ if $\bo{m}\leq \bo{n}$ and if there exists $i$ such that $m_i< n_i$.
Each sequence $\theta_{{\bf q},{\bf n}}(\bb x)$ will be called a \emph{cyclical permutation} of $\bb x$. 
We say that the system $(\bo r,\bb x)$ has \emph{solution} $\bo m\leq \bo n$ if $\bo r+\bo 1\cdot \bb x_{\bo{m}}=\bo 0$, where $\bb x_{\bo m}:=\big(x^{i,j}_{m_i}\big)_{i,j}$.

\begin{defi}[Good cyclical permutation]\label{defiCodifyingSequencesOfMultitypeForestsWithR}
Let $\bb x\in S_d$, of length  $\bo{n}=(n_1,\ldots, n_d)\in \mathbb{N}^d$ and $\bo{r}=(r_1,\ldots, r_d)\in \mathbb{Z}^d_+$ with $\sum r_i>0$ and $\bo r\leq \bo n$. 
We say that $\bo{\theta_{{\bf q},{\bf n}}}(\bb x)$ is a \emph{good cyclical permutation}, if $\bo n$ is a solution of the system $(\bo{r},\bo{\theta_{{\bf q},{\bf n}}}(\bb x))$, and if there is no smaller solution $\bo{m}<\bo{n}$.
Define $S^{\bo r}_d$ as the subset of $S_d$ whose length $\bo n$ is the smallest solution to the system $(\bo r,\bb x)$. 
\end{defi}

Consider $\bb x\in S^{\bo r}_d$ of length $\bo n\in \na^d$ with $x^{i,i}(n_i)\neq 0$ for every $i\in [d]$. 
Then, Lemma 3.3 in \cite{MR3449255} says that the number of good cyclical permutations of $\bb x$ is $det(-\bb x_{\bo n})$, the determinant of the matrix $-\bb x_{\bo n}$.

We will write $\bo T(\bo{r},\bb x)$ to denote the smallest solution of the system $(\bo r,\bb x)$, that is
\begin{equation}\label{eqnDefinitionT_rx}
\bo T(\bo{r},\bb x):=\min\{\bo{m}\in \bb Z^d_+:r_j+\paren{\bo 1\cdot \bb x_{\bo{m}}}_j=0,\ \forall\ j \},
\end{equation}
The latter is also called a \emph{multidimensional first-passage time}.
For simplicity we write $\bo T_{\bo r}$ when no confusion is possible.
The most important property of $\bo T_{\bo r}$, is that when the system $(\bo r,\bb x)$ has its length $\bo n$ as the smallest solution, then it is the BFW of a multitype forest, which is Theorem 2.7 in \cite{MR3449255}.
That is to say, $\bb x\in S^{\bo r}_d$ are the actual sequences that code a multitype forest with $r_i$ roots type $i$ and $n_i$ individuals type $i$.

Since in most of the cases, we fix the number of roots or number of individuals of each type, we need the following definition.

\begin{defi}[Root-type and individuals-type]
	We say a multitype forest with $d\in \na$ types has \emph{root-type} $\bo{r}=(r_1,\ldots, r_d)\in \z^d_+$, if it has $r_i$ roots of type $i$ for $i\in [d]$, with $\sum r_i>0$. 
	We say it has \emph{individuals-type} $\bo{n}=(n_1,\ldots, n_d)\in\na^d$ if it has $n_i\geq r_i$ individuals of type $i$, for $i\in [d]$.
\end{defi}

\subsection{Simulation of multitype forests with a given degree sequence}

The above setting tells us how to code a multitype forest, and conversely, when a given a path in $S_d$ codes a forest.
Now, we not only construct paths coding a forest, but also, impose that the latter has a given multitype degree sequence.
A \emph{multitype degree sequence} $\bb{s}=(\bo{s}_{i,j},i,j\in[d])$ is a sequence of sequences of non-negative integers $\bo{s}_{i,j}=(n_{i,j}(k);k\in [m_{i,j}]_0)$, with $m_{i,j}\in \z_+$, satisfying:
\begin{enumerate}\label{defiMultivariateDegreeSeq}
	\item For each $i \in [d]$, the sum $\sum_k n_{i,j}(k)$ is independent of $j$ and strictly positive; we denote it by $n_i$,
	\item $n_j=r_j+\sum\limits_{k=1}^{m_{1,j}}k n_{1,j}(k)+\cdots +\sum\limits_{k=1}^{m_{d,j}}k n_{d,j}(k)$, for every $j\in [d]$, with $0\leq r_i\leq n_i$ for every $i$, and $r_i>0$ for some $i$,
	\item $\det(-\bb k)>0$ with $\bb k:=(k_{i,j})$, and  $k_{i,j}:=\sum\limits_{k=1}^{m_{i,j}} k n_{i,j}(k)-n_i\indi{i=j}$. 
\end{enumerate}The value $n_{i,j}(k)$ represents the number of individuals of type $i$ with $k$ children of type $j$, so $n_i$ represents the total number of individuals of type $i$. 
Thus, the total number of vertices is $\# \bb s:=n_1+\cdots +n_d=\sum_k n_{1,j}(k)+\cdots +\sum_k n_{d,j}(k)$ for $j\in [d]$.
Table \ref{tableDefinitionMultitypeDegreeSequences} summarizes the case $d=2$.
More explicitly, the forest in Figure \ref{figBFOandSUbtrees} has multitype degree sequence $\bo{s}_{1,1}=(4,1,1)$, $\bo{s}_{1,2}=(5,1)$, $\bo{s}_{2,1}=(3,2)$ and $\bo{s}_{2,2}=(3,1,1)$.  

\begin{table}[ht]
	\centering 
	\begin{tabular}{c c ||c} 
		$\bo{s}_{1,1}=(n_{1,1}(0),\ldots, n_{1,1}(m_{1,1}))$ & $\bo{s}_{1,2}=(n_{1,2}(0),\ldots, n_{1,2}(m_{1,2}))$ & $n_1=\sum_k n_{1,j}(k)$ \\ [1ex]
		$\bo{s}_{2,1}=(n_{2,1}(0),\ldots, n_{2,1}(m_{2,1}))$ & $\bo{s}_{2,2}=(n_{2,2}(0),\ldots, n_{2,2}(m_{2,2}))$ & $n_2=\sum_k n_{2,j}(k)$ \\[1ex]\hline\hline
		$n_1=r_1+\sum k n_{1,1}(k)+\sum k n_{2,1}(k)$ & $n_2=r_2+\sum k n_{1,2}(k)+\sum k n_{2,2}(k)$ & $n_1+n_2=\# \bb s$ \\ [1ex] 
	\end{tabular}
	\caption{Relations on the degree sequence of a 2-type forest.} 
	\label{tableDefinitionMultitypeDegreeSequences} 
\end{table}
Let $\mathbb{F}_{\bb{s},\bo r}$ be the set of multitype plane forests with degree sequence $\bb{s}$, having root-type $\bo{r}$ and individuals-type $\bo{n}$. 
In this paper we work with the following random forests. 

\begin{defi}[Uniform multitype forests with a given degree sequence]\label{defiMFGDS}
	A \emph{uniform multitype forest with a given degree sequence} (MFGDS), with multitype degree sequence  $\bb{s}$ and having root-type $\bo{r}$, is a multitype forest chosen uniformly at random from $\mathbb{F}_{\bb s, \bo r}$.
	Its law will be denoted by
	$\p_{\bb s, \bo r}$.
\end{defi}

As in the unitype case, we construct the canonical \emph{child sequence} $\bb c=(\bo{c}_{i,j},i,j\in [d])$ from the degree sequence, that is, let $\bo{c}_{i,j}$ be a sequence whose first $n_{i,j}(0)$ entries  are zeros, the next $n_{i,j}(1)$ entries are ones, and so on. 
Let $\sigma_{i,j}$ be any permutation on $[n_i]$, 
and construct $\bb w^b=(w^b_{i,j};i,j\in [d] )$, where 
\begin{equation}\label{eqnBridgesFromAMultitypeDegreeSequence}
	w^b_{i,j}(k)=\sum_{\ell=1}^k\paren{\bo{c}_{i,j}\circ \sigma_{i,j} (\ell)-\indi{i=j}}, \ \ \ \ k\in[n_i]_0.
\end{equation}
\begin{remark}\label{remarkDegreeSequenceIsWellDefined}
	Note that $k_{i,j}=k_{i,j}(\bb w^b):=w^b_{i,j}(n_i)$ does not depend on the permutation, so it is deterministic, justifying the superscript $b$ standing for \emph{bridge}. 
Note also that the system  $(\bo{r},\bb w^b)$ admits $\bo{n}$ as a solution.
\end{remark}

\begin{defi}
Fix any $\bo{n}=(n_1,\ldots, n_d)\in \mathbb{N}^d$ and $\bo{r}=(r_1,\ldots, r_d)\in \mathbb{Z}^d_+$ with $\sum r_i>0$ and $\bo r\leq \bo n$. 
Consider $(k_{i,j}; i,j\in [d])$ with $-k_{j,j}=r_j+\sum_{i \neq j}k_{i,j}$, and $k_{i,j}\in \mathbb{Z}_+$. 
We denote by $\mathbb{B}_{\bb s, \bo r}\subset S_d$ the set of bridges $\bb w\in S_d$, where $(\bo r,\bb w)$ has $\bo n$ as solution, ending at the same values $k_{i,j}=w_{i,j}(n_i)$, and with paths having the multitype degree sequence $\bb s$ (that is, its increments constructed from $\bb s$).
\end{defi}

We define a Vervaat-type transformation of $\bb w^b$, given by choosing uniformly at random a good cyclical permutation from all the $\det(-\bb k)$ good cyclical permutations of $\bb w^b$.
After that, the algorithm is similar to the unidimensional case.
More formally, consider the set of \emph{Permutations} of $\bb w^b$ that are \emph{Cyclical} and produce a forest, that is
\[
PC(\bb w^b):=\{\bo{q}:0\leq \bo{q}\leq \bo{n}-\bo{1}_d \mbox{ such that }  \theta_{{\bf q,n}}(\bb w^b) \mbox{ is a good cyclical permutation}\}.
\]Order the elements of $PC(\bb w^b)$ as $\bo{q}_1<\bo{q}_2<\cdots <\bo{q}_{det(-\bb k)}$ using the lexicographic order. 
Thus, we can define a function $\phi: \mathbb{B}_{\bb s, \bo r}\times [det(-\bb k)]\mapsto S_d$ that maps an integer $u\in [det(-\bb k)]$ and a path $\bb w^b$ to the $u$-th good cyclical permutation $\theta_{{\bo{q}}_u,\bo{n}}(\bb w^b)$.

\begin{defi}[Multidimensional Vervaat Transform]\label{defiMultidimensionalVervaatTransform}
	Given $\bo{r}$ and $\bb{s}$, fix any $\bb w^b\in\mathbb{B}_{\bb s, \bo r}$ and any $u\in [\det(-\bb k)]$. 
	Define the multidimensional Vervaat transform $V(\bb w^b, u)$ of $\bb w^b$ at $u$ as
	\[
	V(\bb w^b, u):=\theta_{\phi(\bb w^b,u),{\bf n}}.
	\]
\end{defi}
In words, given a bridge $\bb w^b$ of length $\bo n$ that has $\bo n$ as solution to $(\bo r,\bb w^b)$, the path $V(\bb w^b, u)$ codes a multitype forest with root-type $\bo r$ and individuals-type $\bo n$ for any $u\in [\det(-\bb k)]$.
Furthermore, as $u$ ranges over $[\det(-\bb k)]$, one recovers the complete set of multitype forests obtainable from $\mathbf{w}^b$ via cyclical permutations. 
The previous definition generalizes to the multidimensional case the result from \cite{MR0515820}. 

Having established the deterministic mapping from a degree sequence to its corresponding coding paths, we now present an algorithm for simulating MFGDS. The efficiency of this approach lies in its reduction of the sampling problem: rather than selecting a forest uniformly at random from the entire combinatorial space, the randomness is restricted to $d^2$ uniform permutations and a single uniform random variable.

Let $\bm \pi=(\pi_{i,j},i,j\in [d])$ be independent uniform permutations, where $\pi_{i,j}$ takes values on $[n_i]$, and let $U$ be an independent uniform variable on $[\det(-\bb k)]$, with $\bb k$ defined as in Remark \ref{remarkDegreeSequenceIsWellDefined}. 
Define the process $\bb W^{b,\bb s}=(W^{b,\bb s}_{i,j},i,j\in [d])$ using 
the child sequence $\bb c=\paren{c_{i,j},i,j\in [d]}$ of $\bb s$ and the uniform permutations $(\pi_{i,j})$, as in \eqref{eqnBridgesFromAMultitypeDegreeSequence}.

\begin{teo}[Uniform multitype forest with a given degree sequence]\label{teoConstructionOFUniformMtypeForestWithGDS}
	Fix the degree sequence $\bb s$ of a multitype forest having root-type $\bo{r}$ and individuals-type $\bo{n}$.
	Let $\bb W^{\bb s}$ be the BFW coding a forest taken uniformly at random from $\mathbb{F}_{\bb s, \bo r}$. 
	On the other hand, consider $\bb W^{b,\bb s}$ as defined above.	
	Then 
	\begin{equation*}
		V(\bb W^{b,\bb s},U)\stackrel{d}{=}\bb W^{\bb s}.
	\end{equation*}
\end{teo}The proof is a combinatorial argument, and relies strongly on the fact that $\bb k$ is deterministic. 
This theorem generalizes Lemma 7 in \cite{MR3188597}, to the multidimensional setting.
We derive that $|\mathbb{F}_{\bb s, \bo r}|$, the number of multitype forests with a given degree sequence $\bb s$ and root-type $\bo r$ (cf. Theorem 3.3.2 in \cite{nguyen:tel-01461615}) is
\begin{equation*}
	|\mathbb{F}_{\bb s, \bo r}|=\frac{\det(-\bb k)}{\prod n_i}\prod\prod {n_i\choose \bo s_{i,j}}.
\end{equation*}

In Figure \ref{figMFGDS32000} we depict a simulation of a MFGDS.

\begin{figure}
	\centering
	\includegraphics[width=.75\textwidth]{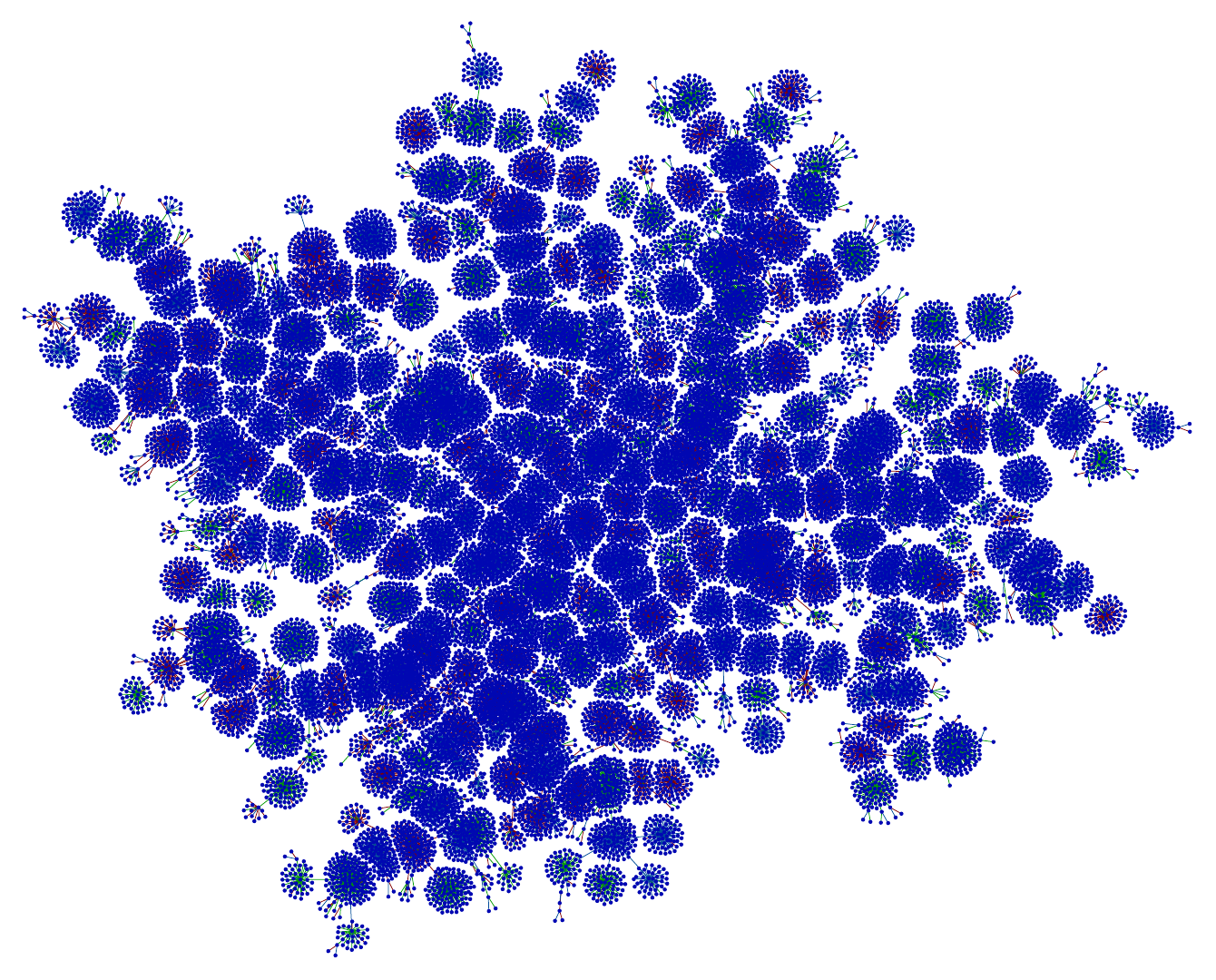}
	\caption{MFGDS with three types, having 32\,000 vertices. }\label{figMFGDS32000}
\end{figure}

\subsection{Simulation of multitype Bienaym\'e-Galton-Watson forests conditioned on the total size of each type}

Consider $\bm \mu=(\mu_1,\ldots, \mu_d)$, where each $\mu_i$ is a distribution on $\mathbb{Z}^d_+$.
A MBGW forest with $d$-types and root-type $\bo{r}$, is a random multitype forest in which each of the $r_i$ roots have children according to $\mu_i$ for each $i\in [d]$, and for every $j$, each children type $j$ has children independently of the other individuals in its generation, according to $\mu_j$.
We call $\bm \mu$ the \emph{progeny distribution} or \emph{offspring distribution} of the forest.
The formal definition is the following.

\begin{defi}
A multitype Bienaym\'e-Galton-Watson forest $\mathcal{F}_{\mathbf{r}}$ with $d$ types, root-type $\mathbf{r} \in \mathbb{Z}^d_+ \setminus \{\mathbf{0}\}$, and offspring distribution $\bm{\mu} = (\mu_i)_{i \in [d]}$,  is a random variable under $\bo P$ taking values in the set $\mathbb{F}_{\mathbf{r}}$ of  finite multitype forests with root-type $\bo r$. Its law is characterized by the property that for every fixed forest $\mathbf{f} \in \mathbb{F}_{\mathbf{r}}$, 
\[
\bo P(\mathcal{F}_{\mathbf{r}} = \mathbf{f}) = \prod_{u \in \mathbf{f}} \mu_{{\rm type}_{\mathbf{f}}(u)} \big( \kappa_1(u), \ldots, \kappa_d(u) \big).
\]We denote by $\bo O_{\bo r}:=(\#_1 \mathcal{F}_{\bo r},\ldots, \#_d \mathcal{F}_{\bo r})$ the vector coding the \emph{number of individuals of each type}. 
\end{defi}

As in Theorem 1.2 in \cite{MR3449255}, we consider MBGW forests satisfying the following.
For $i,j\in [d]$, let $m_{i,j}=\sum_{\bo z\in \z_+^d}z_j\mu_i(\bo z)$ be the mean number of children type $j$ given by an individual type $i$, and set $\bb M=(m_{i,j})_{i,j}$ as the mean matrix of the MBGW forest. 
Whenever $\bb M$ is irreducible and finite, by the Perron-Frobenius Theorem (see \cite[Chapter V.2]{MR2047480}), it has a unique eigenvalue which is simple, positive and with maximal modulus. 
We say in such case that the MBGW forest is \emph{irreducible}. 
If the unique eigenvalue is smaller, equal or greater than one, then we say the forest is \emph{subcritical, critical or supercritical}, respectively.
The forest is \emph{non-degenerate} if individuals have exactly one offspring with probability different from one.

Let $\bb X:=(\bo X^{(1)}, \ldots, \bo X^{(d)})$ be $d$ random walks on $\mathbb{Z}^d$ starting at zero under $\p$, where
\begin{equation}\label{eqnDefinitionLawOfXToi}
	\proba{\bo X^{(i)}_1=\bo{k}}=\mu_i(\bo{k}+\bo{e}_i)\qquad i\in [d],
\end{equation}$\bo{k}$ is a vector with entries in $\mathbb{Z}_+$ except at position $i$, which takes values on $\mathbb{Z}_+\cup\{-1\}$, and $\bo{e}_i$ is the vector with zeros except a one at position $i$. 
We will write $\bo X^{(i)}=(X^{i,1},\ldots, X^{{i,d}})$.
Let us denote by $\# \bo n=\sum_{\ell\in [d]} n_\ell$ and $\# \bo r=\sum_{\ell\in [d]} r_\ell$. 
Our hypotheses are the following:

\begin{description}
	\item[H1]\label{descriptionCMGIndependence} For every $i\in [d]$, the law $\mu_i$ has independent components, that is
	\begin{equation*}
		\mu_i\paren{\bo{k}+\bo{e}_i}=\prod_{j}\proba{X^{i,j}_{1}=k_{j}}\ \ \ \ \ \ \bo{k}=(k_1,\ldots, k_d), k_j\in \mathbb{Z}_+\mbox{ for $j\neq i$ and }k_i\in \mathbb{Z}_+\cup\{-1\}.
	\end{equation*}
	
	\item[H2]\label{descriptionCMGConvolution} For $\mathbf{r} \in \mathbb{Z}^d_+ \setminus \{\mathbf{0}\}$, $\bo n\in \na^d$, and $i,j\in [d]$, with $i\neq j$
	\[
	\esp{X^{i,j}_{n_i} \;\middle|\; \sum_{\ell\in [d]}X^{\ell,j}_{n_\ell} = -r_j}
	=\frac{n_i(n_j-r_j)}{\# \bo n}.
	\]
\end{description}

Since the above conditional expectations will be important in this paper, for $i, j \in [d]$ with $i \neq j$, let us define
	\[
	\widehat{k}_{i,j}(\bo{n}, \bo{r}) := \esp{X^{i,j}_{n_i} \;\middle|\; \sum_{\ell=1}^d X^{\ell,j}_{n_\ell} = -r_j}.
	\]
Let $\widehat{\bb k}(\bo{r}, \bo{n})$ be the $d \times d$ matrix with off-diagonal entries $\widehat{k}_{i,j}(\bo{n}, \bo{r})$ and diagonal entries $-\widehat{k}_{j,j}(\bo{n}, \bo{r}) = r_j + \sum_{i \neq j} \widehat{k}_{i,j}(\bo{n}, \bo{r})$. 

Now we state an algorithm to simulate MBGW forests conditioned by its types, which is our main result.
The idea behind it relies on the fact that 
the BFW of a MBGW forest with root-type $\bo r$ and offspring distribution $\bm \mu$ conditioned on $\bo O_{\bo r}=\bo n$, say $\bb W_{\bo n,\bo r}$, satisfies 
\[
\bb W_{\bo n,\bo r}\stackrel{d}{=}\big(\left.\bb X\right| \bo T_{\bo r}=\bo n\big),
\]where the random walk $\bb X$ is constructed as in \eqref{eqnDefinitionLawOfXToi}. 
This is one of the consequences of
Theorem 2.7 in \cite{MR3449255}. 
The above relationship is used, together with an acceptance-rejection method (see Algorithm \ref{algorithm8}) to simulate in a fast way such conditioned MBGW forests using only multinomial random variables, uniform permutations and a uniform random variable.

\begin{teo}[Simulation of CMBGW$(\bo n,\bo r)$ forests]\label{teoSimulationOfCMBGWIntro}
Consider an offspring distribution $\bm{\mu} = (\mu_i)_{i \in [d]}$ from an irreducible, non-degenerate, (sub)critical MBGW forest, with root-type $\bo{r} \in \mathbb{Z}^d_+\setminus \{\bo 0\}$, and consider $\bo{n}\in \na^d$ with $\bo r\leq \bo n$.
Assume also that  \hyperref[descriptionCMGIndependence]{\bo{H1}} is satisfied.  
Generate independent multinomial vectors $(N_{i,j}(0),N_{i,j}(1),\ldots)$ with parameters $(n_i;\mu_{i,j}(0),\mu_{i,j}(1),\ldots )$, until the first time that $r_j+\sum_i\sum_k k N_{i,j}(k)=n_j$ for every $j$. 
	Denote by $\bb{S}$ the multitype degree sequence obtained, and let $V(\bb W^{b,\bb S},U)$ be the BFW generated by Theorem \ref{teoConstructionOFUniformMtypeForestWithGDS} using the degree sequence $\bb{S}$. 
	Then, for $\bb X$ a random walk constructed as  in \eqref{eqnDefinitionLawOfXToi}, we have
	\begin{equation*}
		\p\paren{V(\bb{W}^{b,\bb S},U)=\bb w} =  \frac{\det\paren{-\widehat{\bb k}(\bo{r}, \bo{n})}}{\det(-\bb k)}\p\paren{\bb X=\bb w|\bo T_{\bo{r}}=\bo n},
	\end{equation*}for every path $\bb w$ coding a multitype forest $\bo f$ with root-type $\bo{r}$, individuals-type $\bo{n}$, multitype degree sequence $(n_{i,j},i,j\in [d])$ and with $k_{i,j}=\sum k n_{i,j}(k)-n_i\indi{i=j}$. 
If in addition \hyperref[descriptionCMGConvolution]{\bo{H2}} is satisfied, the determinant in the numerator evaluates exactly to $\frac{\# \bo r}{\# \bo n} \prod_{i=1}^d n_i$, yielding 
	\begin{equation*}
		\p\paren{V(\bb{W}^{b,\bb S},U)=\bb w} =  \frac{\# \bo r}{\# \bo n}\frac{\prod n_i}{\det(-\bb k)}\p\paren{\bb X=\bb w|\bo T_{\bo{r}}=\bo n}.
	\end{equation*}
\end{teo}

This theorem generalizes the results in \cite{MR2888318}. 
We emphasize that our setup is general: we do not assume that the offspring distributions $(\mu_i)_{i \in [d]}$ are identical, nor do we require identical distribution among the components within each $\mu_i$ or higher order moments. 
Furthermore, in Lemma \ref{lemmaExchangeabilityH2} we give a simple condition under which \hyperref[descriptionCMGConvolution]{\bo{H2}} holds. 

Other algorithms for the simulation of MBGW trees have been explored in the literature; for instance, we refer the reader to \cite{MR1136990, MR1331596, cstefuanescu1998simulation,MR4649394}.

Furthermore, we demonstrate that our algorithm is significantly more efficient than a naive rejection-based method, which simulates paths until one satisfying the required properties is obtained. 
To this end, we derive the computational complexity of generating a CMBGW$(\mathbf{n},\bo r)$ forest and show that it is strictly lower than that of the naive approach. 
This complexity analysis is established by evaluating the expected number of operations required to generate the underlying random variables and analyze their asymptotic behaviour.
To state it, denote by $\mu^{*\bo{n}}_{j}(-r_j):= \proba{\sum_{\ell\in[d] }X^{\ell, j}_{n_\ell}=-r_j}$ and by $\tau_{i,j}(n_i):=\esp{\max_{\ell\in [n_i]}\paren{X^{i,j}_{\ell}-X^{i,j}_{\ell-1}}}$. 
We assume we operate a hypothetical computer, called Random Access Machine (RAM model).  
It takes one time step to compute basic logical and arithmetic operations, to simulate a uniform random variable on $(0,1)$, to access the memory, and assume that independent copies of $X^{i,j}_1$, $i,j\in [d]$, 
can be generated in expected time bounded by one. 
The following complexity time is based in Algorithm \ref{algorithm8}. We remark that Algorithm \ref{algorithm8} only relies on an offspring distribution which is irreducible, non-degenerate, (sub)critical MBGW forest, and \hyperref[descriptionCMGIndependence]{\bo{H1}}.

\begin{algorithm}
	\caption{Generate a CMBGW($\bo{n},\bo r$) forest $\mathcal{F}$}
	\label{algorithm8}
	\begin{algorithmic}[1]
		\STATEx {\bf Input: }A distribution $\bm \mu$, $\bo r\in\mathbb{Z}^d_+\setminus \{\bo 0\}$ and $\bo n\in \mathbb{N}^d$ with $\bo r\leq \bo n$
		\STATEx {\bf Output: }A multitype forest with law $\bo P(\cdot \mid \bo O_{\bo r}=\bo n)$
		\STATE Generate independent multinomial vectors $\bo{S}_{i,j}=(N_{i,j}(0),N_{i,j}(1),\ldots)$ with parameters $(n_i;\mu_{i,j}(0),\mu_{i,j}(1),\ldots )$ 
		\STATE Let $K_{i,j}$ be the last non-zero component of $\bo{S}_{i,j}$, that is $N_{i,j}(K_{i,j})>0$ and $N_{i,j}(k)=0$ for $k>K_{i,j}$
		\STATE Define $\Xi_j:=\Xi(\bo{S}_{i,j},i\in [d])=r_j+\sum_i\sum_k k N_{i,j}(k)$ for every $j$
		\IF{$\Xi_j \neq n_j$ for some $j$} \STATE{\textbf{Reject} and repeat from Step 1} \ENDIF
		\STATE Define $k_{i,j}:=\sum k N_{i,j}(k)-n_i\indi{i=j}$ and construct the matrix $\bb k = (k_{i,j})$
		\STATE Generate an independent uniform variable $V$ on $[0,1]$
		\IF{$V > \frac{\det(-\bb k)}{\prod n_i}$} \STATE{\textbf{Reject} and repeat from Step 1} \ENDIF
        \STATE Apply Algorithm \ref{algorithmUMFGDS} to the accepted degree sequence $\paren{(N_{i,j}(0),\ldots, N_{i,j}(K_{i,j}));i,j\in [d]}$, obtaining a multitype forest $\mathcal{F}$ with breadth-first walk distributed as $V(\bb W^b,U)$, where $U$ is a uniform r.v. in $[\det(-\bb k)]$. 
	\end{algorithmic}
\end{algorithm}

\begin{teo}[Expected complexity time]\label{teoIntroComplexityTime}
Consider a MBGW forest which is irreducible, non-degenerate, (sub)critical, with root-type $\bo{r} \in \mathbb{Z}^d_+\setminus \{\bo 0\}$, and consider $\bo{n}\in \na^d$ with $\bo r\leq \bo n$.
Assume also that  \hyperref[descriptionCMGIndependence]{\bo{H1}} and \hyperref[descriptionCMGConvolution]{\bo{H2}} are satisfied.   
	Assuming a RAM model of computation, a CMBGW$(\bo n,\bo r)$ forest can be generated in expected time bounded above by a constant times
	\begin{equation*}
	\frac{\# \bo n}{\# \bo r} \paren{ \frac{d^2 + \sum_{i,j}\tau_{i,j}(n_i)}{\prod_{j=1}^d \mu^{*\bo{n}}_{j}(-r_j)} + d^3 } + \# \bo n \, \E\bra{ \prod_{i=1}^d (-X^{i,i}_{n_i}) \;\middle|\; B^{(\bo n)} },
	\end{equation*}
	where $B^{(\bo n)}$ is the event that the degree sequence is accepted in Step 9.
	In particular, 	$\# \bo n \to \infty$, $n_i = \Theta(\# \bo n)$ for all $i \in [d]$, and 
 if $\E\Big(\big(X^{i,j}_1\big)^2\Big)$ is finite for every $i,j$, then 
 the expected time complexity is  of the order
	\begin{equation*}
	o\paren{ (\# \bo n)^{(d+3)/ 2}}.
	\end{equation*}
\end{teo}

For comparison, the expected complexity time of a na\"ive method is (see Subsection \ref{subsectionComplexityNaive})
\[
\Theta\left( (\# \bo n)^{\max(d/2 + 2, \, d + 1)} \right).
\]

We use the notation $\Theta(\cdot)$ to denote the asymptotic order of magnitude, indicating that a term is bounded both above and below by its argument up to universal multiplicative constants.
Furthermore, we write $g(\bo n)=o(f(\bo n))$ to signify that the ratio $g(\bo n)/f(\bo n)$ vanishes as $\|\bo n\| \to \infty$.

The asymptotic behavior of the expected time complexity in theorem \ref{teoIntroComplexityTime}, is derived from the precise control of the probabilities $\mu^{*\bo{n}}_{j}(-r_j)$ and $\p\big(B^{(\bo n)}\big)$.
The latter satisfies $\p\big(B^{(\bo n)}\big)=\proba{\bo T_{\bo r}=\bo n}$, a relation we establish in the sequel.
To state the next result, which provides the asymptotic computation of these probabilities, recall that $m_{\ell,j} = \E\big(X^{\ell,j}_1 + \indi{\ell=j}\big)$.
We further define the mean and variance of the $j$-th component of the field at $\bo n$ as
$ m^{(\bo n)}_j := \esp{(\bo 1\cdot \bb X_{\bo n})_j} = \sum_{\ell=1}^d n_\ell m_{\ell,j} - n_j $
and $ \sigma^{(\bo n)}_j := {\rm Var}\paren{(\bo 1\cdot \bb X_{\bo n})_j} = \sum_{\ell=1}^d n_\ell \sigma^2_{\ell,j} $. 
Finally, let $D \subseteq [d]$ be the set of indices $\ell$ for which the law of $X^{\ell,j}_1$ is non-degenerate, and let $H_\ell$ denote its corresponding maximal span.

\begin{lemma}[Local Limit Theorem for multidimensional first-hitting times]\label{lemmaLLTTargetSums}
	Assume that for every $i,j \in [d]$, the offspring marginal distribution $\mu_{i,j}$ has finite variance $\sigma^2_{i,j} < \infty$, and that for each $j$, the greatest common divisor of $(H_\ell;\ell\in D)$ is 1. 
	Assume that as $\# \bo n \to \infty$, $n_i = \Theta(\# \bo n)$ for all $i \in [d]$. 
	If there exists a constant $c > 0$ such that $|-r_j - m^{(\bo n)}_j| \le c \sqrt{\# \bo n}$ for all $j \in [d]$, then
\[
\sup_{r\in \mathbb{Z}}\left|
 	\sqrt{\sigma^{(\bo n)}_j}\mu^{*\bo n}_j(r)-\frac{1}{\sqrt{2\pi }} \exp\paren{ - \frac{(r - m^{(\bo n)}_j)^2}{2\sigma^{(\bo n)}_j} }\right| \underset{\#\bo n\to\infty}{\longrightarrow}0.
\]
	Consequently, under Hypothesis \hyperref[descriptionCMGIndependence]{\bo{H1}} and \hyperref[descriptionCMGConvolution]{\bo{H2}}, for any integer $R>0$ fixed
\[
\sup_{\bo r\in [R]^d}\left|
 	\proba{\bo T_{\bo r} = \bo n}\prod_{j\in[d]}\sqrt{\sigma^{(\bo n)}_j}-\frac{\# \bo r}{\# \bo n}\frac{1}{(2\pi)^{d/2}} \exp\paren{ -\sum_{j\in[d]} \frac{(r_j + m^{(\bo n)}_j)^2}{2\sigma^{(\bo n)}_j} }\right| \underset{\#\bo n\to\infty}{\longrightarrow}0.
\]In particular, for fixed $\bo r$, we have
\begin{equation*}
	\proba{\bo T_{\bo r} = \bo n} = \Theta\paren{ \frac{1}{(\# \bo n)^{d/2 + 1}} }.
\end{equation*}
\end{lemma}

In Example~\ref{lemmaExplicitConstruction}, we provide an explicit construction of an offspring distribution and a sequence of sizes $\bo n$ that satisfy the condition $|-r_j - m^{(\bo n)}_j| \le c \sqrt{\# \bo n}$ for all $j \in [d]$.
The proof of Lemma~\ref{lemmaLLTTargetSums} follows directly from the local limit theorem for independent variables \cite{MR388499}, combined with the joint law governing the number of individuals of each type.
A detailed description of this joint law is provided in the following subsection.

\subsection{Law of the number of individuals of each type}

The proof of Theorem \ref{teoSimulationOfCMBGWIntro} is established by combining Theorem \ref{teoConstructionOFUniformMtypeForestWithGDS} with the derivation of the law of the number of individuals of each type $\mathbf{O}_{\mathbf{r}} = (\#_1 \mathcal{F}_{\mathbf{r}}, \ldots, \#_d \mathcal{F}_{\mathbf{r}})$. 
The latter provides a multitype generalization of the classical Otter--Dwass formula \cite{MR0030716, MR0253433} and Kemperman's formula \cite{MR0038045}, which are interesting on its own. 
It builds upon Theorem 1.2 in \cite{MR3449255}. 
This representation utilizes vectors $\mathfrak{e} = (\mathfrak{e}_1, \ldots, \mathfrak{e}_d) \in \mathbb{F}_d^{\text{elem.}}$ that encode \emph{elementary forests}. Specifically, each $\mathfrak{e}$ corresponds to a mapping $\mathfrak{e}: [d] \to [d]_0$ such that $\mathfrak{e}_j \neq j$ for all $j \in [d]$ (we refer the reader to Definition \ref{defiElementaryForest} for a formal characterization).

\begin{description}
	\item[H3]\label{descriptionCMGDeterministicOffDiagonal} For every $i,j\in [d]$, with $i\neq j$ there exists $C_{i,j}\in \na$ such that $X^{i,j}_{n_i}=C_{i,j}n_i$. 

    \item[H4]\label{descriptionCMGProportional} \textbf{Proportional Offspring:} There exist weights $w_1,\dots, w_d>0$ such that for all $i \neq j$:
	\begin{equation*}
		\esp{X^{i,j}_{n_i} \;\middle|\; \sum_{\ell=1}^d X^{\ell,j}_{n_\ell} = -r_j} = \frac{w_in_i}{\sum_{\ell=1}^d w_\ell n_\ell}(n_j-r_j).
	\end{equation*}
	(Note that \hyperref[descriptionCMGConvolution]{\bo{H2}} is a special case of \hyperref[descriptionCMGProportional]{\bo{H4}} when $w_i = 1$ for all $i$. This hypothesis naturally covers branching processes where the expected offspring rates factorize as $\lambda_{i,j} = w_i c_j$).

	\item[H5]\label{descriptionCMGProgressive} \textbf{Hierarchical (Triangular) Dependency:} The types can be ordered such that for every $i > j$, individuals of type $i$ almost surely cannot produce children of type $j$. This forces $X^{i,j}_{n_i} = 0$ almost surely for $i > j$, which models hierarchical or directional systems like irreversible mutations, progressive disease stages, or cell differentiation.
\end{description}

\begin{teo}[Law of the number of individuals of each type of a MBGW forest]\label{teoLawOfTheTotalPopulationByTypes}
	Consider an irreducible, non-degenerate, (sub)critical MBGW forest, with root-type $\bo{r} \in \mathbb{Z}^d_+ \setminus \{\mathbf{0}\}$, and consider $\bo{n}\in \mathbb{N}^d$ with $\bo{r}\leq \bo{n}$.
	Suppose that \hyperref[descriptionCMGIndependence]{\bo{H1}} is satisfied. 
Then the joint law is
	\begin{equation}\label{eqnGeneralDeterminantFormula}
		\bo P\paren{\bo{O}_{\bo{r}}=\bo n} =\proba{\bo T_{\bo r}=\bo n}= \frac{\det\paren{-\widehat{\bb k}(\bo{r}, \bo{n})}}{\prod_{i=1}^d n_i} \prod_{j=1}^d \proba{\sum_{\ell=1}^d X^{\ell,j}_{n_\ell} = -r_j}.
	\end{equation}
	
	This formula simplifies under specific structural hypotheses. If $\bb X$ satisfies \hyperref[descriptionCMGConvolution]{\bo{H2}}, then
	\begin{equation}\label{eqnOterDwassGeneralizationH2}
		\bo P\paren{\bo O_{\bo{r}}=\bo n}=\frac{\# \bo r}{\# \bo n}\prod_{j=1}^d\proba{\sum_{\ell=1}^d X^{\ell,j}_{n_\ell}=-r_j}.
	\end{equation}
	If $\bb X$ satisfies \hyperref[descriptionCMGDeterministicOffDiagonal]{\bo{H3}}, we trivially recover the exact determinant of $\bb C(\bo{r,n})$:
	\begin{equation}\label{eqnDriftOutsideDiagonalLevyField}
		\bo P\paren{\bo{O}_{\bo{r}}=\bo n} =\frac{\det\paren{-\bb C(\bo{r,n})}}{\prod_{i=1}^d n_i}\prod_{j=1}^d \proba{\sum_{\ell=1}^d X^{\ell,j}_{n_\ell}=-r_j},
	\end{equation}
	where $\bb C(\bo{r,n})$ has off-diagonal entries $C_{i,j} n_i$ and diagonal entries $-r_j - \sum_{k \neq j} C_{k,j} n_k$.
	
	Furthermore, if $\bb X$ satisfies \hyperref[descriptionCMGProportional]{\bo{H4}}, the matrix of conditional expectations admits a rank-1 update structure. By the Matrix Determinant Lemma, the determinant simplifies algebraically into a strict closed form:
	\begin{equation}\label{eqnRankOneSimplification}
	    \bo P\paren{\bo{O}_{\bo{r}}=\bo n} = \frac{\sum_{\ell=1}^d w_\ell r_\ell}{\sum_{\ell=1}^d w_\ell n_\ell} \prod_{j=1}^d \proba{\sum_{\ell=1}^d X^{\ell,j}_{n_\ell}=-r_j}.
	\end{equation}
	
	If instead $\bb X$ satisfies \hyperref[descriptionCMGProgressive]{\bo{H5}}, the matrix $-\widehat{\bb k}(\bo{r,n})$ is upper triangular, and the determinant collapses completely to the product of its diagonal terms:
	\begin{equation}\label{eqnDAGSimplification}
	    \bo P\paren{\bo{O}_{\bo{r}}=\bo n} = \frac{1}{\prod_{i=1}^d n_i} \prod_{j=1}^d \left( r_j + \sum_{i < j} \widehat{k}_{i,j}(\bo{n}, \bo{r}) \right) \proba{\sum_{\ell=1}^d X^{\ell,j}_{n_\ell}=-r_j}.
	\end{equation}
\end{teo}

Observe that in the single-type case ($d=1$), formulas \eqref{eqnGeneralDeterminantFormula} in \eqref{eqnOterDwassGeneralizationH2} coincide with the classical Otter--Dwass formula. 
The importance of \eqref{eqnOterDwassGeneralizationH2} is that it allows us to obtain explicit computations. 

For some other results concerning the joint distribution of the total number of individuals by type, we refer the reader to \cite{MR3285942}, \cite[Lemma 3.8]{MR3803914}, and \cite[Lemma 4.7]{MR4586227}.

As a simple corollary, we obtain the following two results. 

Let $\mathbb{B}_{\bo n, \bo r}$ be  the set of all multidimensional paths of length $\bo n$ such that $\Xi_j = n_j$ for all $j \in [d]$, as defined in Step (3) of Algorithm \ref{algorithm8}. 
That is,  $\mathbb{B}_{\bo n, \bo r}$ is the set of all multidimensional bridges of size $\bo n$, such that its degree sequence satisfies $r_j+\sum_i\sum_k k N_{i,j}(k)=n_j$.
Consider $F$ with domain the set of multidimensional paths in $S_d$ of length $\bo n$, with values in $\mathbb{R}$, and invariant under cyclic shifts. 
We write $F_{\bo n}(\bb X):=F(X^{i,j}_{\ell_i};\ell_i\leq n_i,i\in [d])$. 
Now we relate the law of the multidimensional first-hitting times with bridges.
It is an extension of classical results (see \cite{MR3651047}, Lemma 11).

\begin{coro}[Decoupling the determinant]\label{coroRelatingFirstHittingTimeAndBridges}
If $\bb X$ is a random walk with law as in \eqref{eqnDefinitionLawOfXToi}, such that $\bo n$ is a solution of $(\bo r,\bb X)$ with positive probability, then 
\[
\E\paren{F_{\bo n}(\bb X);\bo T_{\bo r}=\bo n}=\E\paren{F_{\bo n}(\bb X)\frac{{\rm det}(-\bb k(\bb X))}{\prod n_i};\bb X\in \bb B_{\bo n,\bo r}},
\]where $\bb k(\bb X)$ is constructed as $k_{i,j}=\sum k N_{i,j}(k)-n_i\indi{i=j}$, using the degree sequence $(N_{i,j})$ of $\bb X$. 
Moreover, assume the branching process with offspring distribution $\bm \mu$ is irreducible, non-degenerate, (sub)critical MBGW forest, and that \hyperref[descriptionCMGIndependence]{\bo{H1}} is satisfied. Then
\[
\E\paren{\frac{{\rm det}(-\bb k(\bb X))}{\prod n_i};\bb X\in \bb B_{\bo n,\bo r}}=\frac{\det\paren{-\widehat{\bb k}(\bo{r}, \bo{n})}}{\prod_{i=1}^d n_i} \prod_{j=1}^d \proba{\sum_{\ell=1}^d X^{\ell,j}_{n_\ell} = -r_j}.
\]
\end{coro}
In the continuous setting, Chaumont and Marine proved in Eqn. (4.22) of \cite{MR4193902}, a formula similar to the LHS
\[
\proba{\bo T_{\bo r}\in dt}=\int_{\re^{d(d-1)}_+}\frac{{\rm det}(-(\overline{\bb x}+\bo rI_d))}{\prod t_i}p_{\bo t}(\bb x^{\bo r})\prod_{k\neq j}{\rm d}x_{k,j}{\rm d}\bo t. 
\]In a joint work with Sandra Palau \cite{AngtuncioPalauConvergenceMBGWProcess2026}, we investigate another way to express the law of the (continuous) multidimensional first-hitting time $\bo T_{\bo r}$. 

\begin{coro}[Total Population of a MBGW Forest]\label{coroLawOfTheTotalPopulationInAMBGW}
	Assume the hypotheses of Theorem \ref{teoLawOfTheTotalPopulationByTypes} are satisfied, together with \hyperref[descriptionCMGIndependence]{\bo{H1}} and \hyperref[descriptionCMGConvolution]{\bo{H2}}.
	Fix $n\geq d$ and let $\bo r\in \mathbb{N}^d$ such that $\#\bo r\leq n$. Suppose that for each $j \in [d]$, there exists a sequence of probability mass functions $f_{n}^{(j)}$ (depending only on $n$ and $j$) such that
    $\proba{\sum_{\ell=1}^d X^{\ell,j}_{n_\ell} = -r_j} = f_n^{(j)}(n_j - r_j)$, for all configurations $(n_1,\ldots, n_d)\in\mathbb{N}^d$ with $\sum_\ell n_\ell=n$. 
    Let $Y_n^{(1)}, \dots, Y_n^{(d)}$ be independent random variables taking values in $\mathbb{Z}_+$ with distributions given by $f_n^{(1)}, \dots, f_n^{(d)}$, respectively.
	Let $\#\bo O_{\bo r}=\sum_{\ell\in [d]}O_\ell$ be the total size of the forest.
	Then
	\begin{align*}
	\bo P\paren{\#\bo O_{\bo r}=n} = \frac{\#\bo r}{ n}\proba{\sum_{j=1}^d Y^{(j)}_{n} =  n -  \#\bo r}.
	\end{align*}
\end{coro}

\subsection{Enumerations of some combinatorial multitype forests}
Finally, we apply Theorems \ref{teoSimulationOfCMBGWIntro} and \ref{teoLawOfTheTotalPopulationByTypes} to provide enumerations of plane, labeled, and binary multitype forests with fixed root-type and individuals-type. 
All of the following results satisfy \hyperref[descriptionCMGConvolution]{\bo{H2}}, as we will see later. 
For alternative enumerative results concerning multitype labeled forests, we refer the reader to \cite{MR2359513}, Theorems 5.49 and 5.54.

For the next results, we denote a MBGW forest with root-type $\bo r$ by $\mathcal{G}_{\bo{r},\bo p}$, $\mathcal{P}_{{\bf r},\bm \mu}$ or $\mathcal{B}_{{\bf r},\bo p}$, if independently for each individual type $i$, the offspring distribution is respectively
\begin{enumerate}
	\item geometric with parameter $\bo p=(p_1,\ldots, p_d)\in (0,1]^d$, that is, $\mu_i(k_1,\ldots, k_d)=\prod_j p_j(1-p_j)^{k_j}$ for $k_j\geq 0$,
	\item Poisson with parameters $\bm\mu=(\mu_1,\ldots,\mu_d)\in\re^d_+$, that is, $\mu_i(k_1,\ldots, k_d)=\prod_j e^{-\mu_j}\mu_j^{k_j}/k_j!$ for $k_j\geq 0$, 
	\item Bernoulli with parameter $\bo p=(p_1,\ldots, p_d)\in [0,1]^d$, that is, $\mu_i(k_1,\ldots,k_d)=\prod_j p_j^{k_j/2}(1-p_j)^{1-k_j/2}$ for $k_j\in \{0,2\}$.
\end{enumerate}
Define by $\mathbb{F}^{plane}_{\bo{r,n}}$, $\mathbb{F}^{labeled}_{\bo{r,n}}$ and $\mathbb{F}^{binary}_{\bo{r,n}}$,\label{notationCombinatorialMtypeForests} the set of $d$-type plane, labeled and binary forests having root-type $\bo{r}$ and individuals-type $\bo{n}$. 
Using Theorem \ref{teoLawOfTheTotalPopulationByTypes}, we give in Subsection \ref{subsectionExamplesMultitypeGeoPoiBer} three examples of distributions where the law of a CMBGW$(\bo n,\bo r)$ forest can be computed explicitly. 
This generalizes the constructions given in \cite{MR1630413} (cf. \cite{MR3157803}).

\begin{propo}[Uniform multitype plane forests on $\mathbb{F}^{plane}_{\bo{n,r}}$]\label{propoH2Geometric}
	If $\sum_{j\in [d]} (1-p_j)/p_j\leq 1$, we have
	\begin{equation*}
		\proba{\mathcal{G}_{{\bf r},\bo p}=\bo f|\#\mathcal{G}_{{\bf r},\bo p}=\bo n}=\frac{1}{\frac{\# \bo r}{\# \bo n}\prod_{i\in [d]}{\# \bo n+n_i-r_i-1\choose n_i-r_i}}\ \ \ \ \ \ \forall\,\bo f\in \mathbb{F}^{plane}_{\bo{n,r}}.
	\end{equation*}
\end{propo}

\begin{propo}[Uniform multitype labeled forests on $\mathbb{F}^{labeled}_{\bo{n,r}}$]\label{propoH2Poisson}
	If $\sum_{j\in [d]} \mu_i\leq 1$ and  $\mathcal{P}^*_{\bo{n,r}}$ is $\mathcal{P}_{\bo{n,r}}$ relabeled by $d$ uniform random permutations, one for each type, then
	\begin{equation*}
		\proba{\mathcal{P}_{{\bf r},\bm \mu}^*=\bo f|\#\mathcal{P}_{{\bf r},\bm \mu }=\bo n}=\frac{1}{\frac{\# \bo r}{\# \bo n}\big(\# \bo n\big)^{\# \bo n-\# \bo r}}\ \ \ \ \ \ \forall\,\bo f\in \mathbb{F}^{labeled}_{\bo{n,r}}.
	\end{equation*}
\end{propo}The previous random relabeling of the plane forest  is used to obtain a uniform forest on the set of labeled forests.

\begin{propo}[Uniform multitype binary forests on $\mathbb{F}^{binary}_{\bo{n,r}}$]\label{propoH2Binomial}
Assume $n_i-r_i$ is an even number for every $i\in [d]$ and that $2\sum_{j\in [d]} p_j\leq 1$.
	Then
	\begin{equation*}
		\proba{\mathcal{B}_{{\bf r},\bo p}=\bo f|\#\mathcal{B}_{{\bf r},\bo p}=\bo n}=\frac{1}{\frac{\# \bo r}{\# \bo n}\prod {\# \bo n \choose (n_i-r_i)/2}} \ \ \ \ \ \ \forall\,\bo f\in \mathbb{F}^{binary}_{\bo{n,r}}.
	\end{equation*}
\end{propo}
As a simple corollary of the previous propositions, the denominators in the three formulas, provide us with the number of $d$-type plane, labeled, and binary forest,  with root-type $\bo{r}$ and individuals-type $\bo{n}$, respectively.

The paper is organized as follows: in Section  \ref{sectionConstructionOfMFGDS} we construct MFGDS and prove Theorem \ref{teoConstructionOFUniformMtypeForestWithGDS}. 
Section \ref{sectionLawOfNumberOfIndividualsByTypes} is devoted to prove the joint law of the number of individuals by types in a MBGW forest, which is Theorem \ref{teoLawOfTheTotalPopulationByTypes}.
In that section we also obtain in Corollary \ref{coroLawOfTheTotalPopulationInAMBGW}, the law of the total population in a MBGW forest.
Examples satisfying the hypotheses of Theorem \ref{teoLawOfTheTotalPopulationByTypes} are given in Subsection \ref{subsectionExamplesMultitypeGeoPoiBer}.
In Section \ref{secRelationMFGDScMBGW} we prove that under an independence assumption, the CMBGW forests are mixtures of MFGDS. 
Our main result Theorem \ref{teoSimulationOfCMBGWIntro} and the proof that Algorithm \ref{algorithm8} provides the law of a CMBGW forest, is given in Section \ref{sectionAlgorithmsConditionedRandomForests}.
Finally, the expected complexity time Thoerem \ref{teoIntroComplexityTime} is proved in Section \ref{sectionExpectedComplexity}, where we also show the local limit theorem for multidimensional first-hitting times.

\section{Multitype random forests with a given degree sequence}\label{sectionConstructionOfMFGDS}

Now we are ready to construct a forest taken uniformly at random from $\mathbb{F}_{\bb s,\bo r}$.

\begin{proof}[Proof of Theorem \ref{teoConstructionOFUniformMtypeForestWithGDS}]\label{proofteoConstructionOFUniformMtypeForestWithGDS}
Recall that from a given multitype degree sequence as defined in \ref{defiMultivariateDegreeSeq}, we can construct bridges $\bb w^b$ taking values on $S_d$ as in  \eqref{eqnBridgesFromAMultitypeDegreeSequence}.
Let us prove that the system of equations $(\bo{r},\bb w^b)$ admits $\bo{n}$ as a solution, and that $w^b_{i,i}(n_i)\neq 0$ for every $i\in [d]$.
The latter is needed to apply Lemma 3.3 in \cite{MR3449255}.

The length $\bo n$ is a solution, since
	\begin{equation*}
		r_j+\sum_{i=1}^d w^b_{i,j}(n_i)=r_j-n_j+\sum_{i=1}^d\sum_k k n_{i,j}(k)=0\ \ \ \ \ \ \ \ \forall\, j\in [d].
	\end{equation*}Using the definition, note that 
	\begin{equation*}
		-k_{j,j}=n_j-\sum_k k n_{j,j}(k)=r_j+\sum_i\sum_k k n_{i,j}(k)-\sum_k k n_{j,j}(k)=r_j+\sum_{i \neq j}k_{i,j}.
	\end{equation*}Thus, whenever $n_j>0$ we have $-k_{j,j}=-w^b_{j,j}(n_j)>0$.
	Indeed, if $-k_{j,j}=0$, from the above display we have $r_j=0$ and $k_{i,j}=0$ for any $i\neq j$. 
	This means, there are no roots type $j$ and no individual type $i\neq j$ gives birth to type $j$ individuals, and this implies $n_j=0$.

The conditions on the multitype degree sequence $\bo s$ that there is at least one individual of each type and that the determinant is positive, are imposed to conclude that $\bo s$ is the multitype degree sequence of some forest. 

The above and the Multivariate Cyclic Lemma (Lemma 3.3 in \cite{MR3449255}), implies that from $\bb s$ there are $det(-\bb k)$ forests that can be associated to it, since any good cyclical permutation codes a forest.
	
	Now, define $\bo{s}_{i,j}=(n_{i,j}(k);k\geq 0)$, and write
	\begin{equation*}
		{n_i \choose \bo{s}_{i,j}}:={n_i \choose n_{i,j}(0),n_{i,j}(1),\ldots}.
	\end{equation*}
	
	Fix any bridge $\bb w^b\in \mathbb{B}_{\bb s, \bo r}$.
	From the possible $\prod_i (n_i!)^d$ values taken by the random permutations $(\bo{\pi}_{i,j},i,j\in[d])$, exactly $\prod_j\prod_i\prod_k n_{i,j}(k)!$ form the bridge $\bb w^b$. 
	This is true since, permuting the labels of the $n_{i,j}(k)$ individuals type $i$ having $k$ children type $j$, we obtain the same bridge. 
	This proves the assertion since this is true for every $i,j,k$.
	Therefore
	\begin{equation*}
		\p\paren{\bb W^b=\bb w^b}=\frac{1}{\prod\prod {n_i\choose \bo{s}_{i,j}}}.
	\end{equation*}Now, fix any $i\in [d]$ and $\bb w\in \mathbb{E}_{\bb s, \bo r}$, the set of BFWs coding all forests with multitype degree sequence $\bb s$.
	We now obtain the number of different pairs $(\bb w^b,u)\in \mathbb{B}_{\bb s, \bo r}\times [det(-\bb k)]$ that can be mapped to $\bb w$ using the multidimensional Vervaat transform, as defined in \ref{defiMultidimensionalVervaatTransform}.
	Note that such bridges can only be of the form $\theta_{\bo{q,n}}(\bb w)$, that is, cyclical permutations of $\bb w$.
	If $\bo w^{(i)}$ is the $i$-th row of $\bb w$, by Lemma \ref{lemmaPairsYUMappedToW} in the Appendix, the number of pairs $(\theta_j(\bo w^{(i)}),u)$ that can be mapped to $\bo w^{(i)}$ are exactly $n_i$.
	This being true for every $i$ implies there are $\prod n_i$ unique pairs $(\theta_{\bf q,n}(\bb w),u)$ such that $V((\theta_{\bf q,n}(\bb w),u))=\bb w$.
	Denote such pairs as
	\begin{equation*}
		A(\bb w)=\left\{(\bb w^b_k,u_k)\in \mathbb{B}_{\bb s, \bo r}\times [\det(-\bb k)] :V((\bb w^b_k,u_k))=\bb w,\ k\in \left[ \prod n_i\right]\right\}.
	\end{equation*}

	This implies
	\begin{align*}
		\p\paren{V(\bb W^b,U)=\bb w}& =\p\paren{ (\bb W^b,U)\in A(\bb w)}\\
		& = \sum_{k\in \left[ \prod n_i\right] }\p\paren{ (\bb W^b,U)=  (\bb w^b_k,u_k)}\\
		& =\sum_{k\in \left[ \prod n_i\right] }\frac{1}{\prod\prod {n_i\choose \bo{s}_{i,j}}}\frac{1}{\det(-\bb k)}\\
		& =\frac{1}{\frac{\det(-\bb k)}{\prod n_i}\prod\prod {n_i\choose \bo{s}_{i,j}}}. 
	\end{align*}This concludes the proof since the right-hand side is independent of $\bb w$, so $V(\bb W^b,U)$ is uniform.
\end{proof}

\begin{remark}
	From this lemma we can conclude that the set of plane forests with degree sequence $\bb{s}$ having root-type $\bo{r}$ is
	\begin{equation*}
		|\mathbb{F}_{\bb s, \bo r}|=\frac{\det(-\bb k)}{\prod n_i}\prod\prod {n_i\choose \bo{s}_{i,j}}.
	\end{equation*}
\end{remark}

In our following Algorithm \ref{algorithmUMFGDS}, we use the multidimensional Vervaat transform as defined in page \pageref{defiMultidimensionalVervaatTransform}, to simulate a MFGDS.

\begin{algorithm}
	\caption{Generate uniformly sampled multitype forests with a given degree sequence}
	\label{algorithmUMFGDS}
	\begin{algorithmic}[1]
		\STATEx {\bf Input: }A degree sequence $\bo{ s}_{i,j}=(n_{i,j}(k);k\in [m_{i,j}])$ satisfying $n_j=\sum_k n_{i,j}(k)$ for every $i,j$, and $n_j=r_j+\sum_i\sum_kk n_{i,j}(k)$, for every $j$
		
		\STATEx {\bf Output: }A uniformly sampled multitype forest with multitype degree sequence $\bb s$
		\STATE Generate the vectors $\bo{c}_{i,j}=(c_{i,j}(1),c_{i,j}(2),\ldots,c_{i,j}(n_i))$, with $n_{i,j}(0)$ zeros, $n_{i,j}(1)$ ones, etc., ordered in non-decreasing order $c_{i,j}(k)\leq c_{i,j}(k+1)$ 
		\STATE Generate $\bm \pi_{i,j}=(\pi_{i,j}(1),\ldots, \pi_{i,j}(n_i))$, a uniform random permutation of $[n_i]$, everything independent
		\STATE Define $\bb W^b=(W^b_{i,j};i,j\in [d])$, where
		\begin{esn}
			W^b_{i,j}(k)=\sum_{\ell=1}^k(c_{i,j}\circ \pi_{i,j}(\ell)-\indi{i=j}),\ k\in [n_i]
		\end{esn}satisfying $W^b_{i,i}(0)=0$ and $W^b_{i,i}(n_i)=-k_i$
		\STATE Generate an independent uniform random variable $U$ on $[det (-\bb k)]$, where $k_{i,j}:=\sum k n_{i,j}(k)-n_i\indi{i=j}$
		\STATE Construct the multidimensional Vervaat transform $V(\bb W^b,U)$ of $\bb W^b$ 
		\STATE Generate the multitype forest with breadth-first walk $V(\bb W^b,U)$
	\end{algorithmic}
\end{algorithm}

\section{Law of the number of individuals by types of a MBGW forest}\label{sectionLawOfNumberOfIndividualsByTypes}

In this section we obtain the joint law of $\bo O_{\bo r}=(O_1,\ldots, O_d)$, which codes the number of individuals of each type in a MBGW forest. 
We rely on the following theorem.

\begin{teo}[Theorem 1.2, \cite{MR3449255}]\label{teoChaumontTotalPopulationByTypesandParents}
	Let $Z$ be a $d$-type branching process, which is irreducible, non-degenerate and (sub)critical. 
	For $i\in [d]$, let $O_i$ be the total number of individuals of type $i$, up to the extinction time $T$, and for $i\neq j$, let $A_{i,j}$ be the total number of individuals of type $j$ whose parent is of type $i$, up to time $T$.
	Then, for all integers $r_i$, $n_i$, $k_{i,j}$ such that $r_i\geq 0$ with $\bo{r}>0$, $k_{i,j}\geq 0$ for $i\neq j$, $-k_{j,j}=r_j+\sum_{i\neq j}k_{i,j}$, and $n_i\geq -k_{i,i}$, we have
	\begin{equation}\label{eqnPopulationByTypesAndParents}
	\begin{split}
	& \bo P\paren{O_i=n_i,i\in [d],A_{i,j}=k_{i,j}, i,j\in [d],i\neq j}\\
	& = \frac{det(-\bb k)}{\bar{n}_1\cdots \bar{n}_d}\prod_1^d\mu_i^{*n_i}\paren{k_{i1},\ldots ,k_{i(i-1)},n_i+k_{i,i},k_{i(i+1)},\ldots, k_{id}}, 
	\end{split}\end{equation}
	where ${{\bf r}}=(r_1,\ldots, r_d)$, $\bar{n}_i=n_i\vee 1$ and $(-\bb k)_{i,j\in [d]}$ is the matrix with entries $(-k_{i,j})$ to which we remove row $i$ and column $i$, for every $i$ such that $n_i=0$. 
\end{teo}

Let us give a hint on how to derive the law of the population by types for a 2-type BGW forest, having root-type $\bo{r}$ and individuals-type $\bo{n}$.
Recall the hypothesis \hyperref[descriptionCMGIndependence]{\bo{H1}} about the independence in the components of $\mu_i$ of Theorem \ref{teoLawOfTheTotalPopulationByTypes}.
Recalling the definition of $(X^{i,j},i,j\in [d])$ in Equation \eqref{eqnDefinitionLawOfXToi}, from Theorem \ref{teoChaumontTotalPopulationByTypesandParents}, summing over all the possible values of $\bb k$ we have 
\begin{equation}\label{eqnTotalProbabilityForTwoTypes}
\begin{split}
& \bo P\paren{\bo O_{\bo r}=\bo n}\\
& = \sum_{i=0}^{n_1-r_1}\sum_{j=0}^{n_2-r_2}\frac{r_1r_2+r_1j+r_2i}{n_1n_2}\proba{(X_{n_1}^{1,1},X_{n_1}^{1,2})=\paren{-r_1-i,j}}\proba{(X_{n_2}^{2,1},X_{n_2}^{2,2})=\paren{i,-r_2-j}}\\
& = \sum_{i=0}^{n_1-r_1}\sum_{j=0}^{n_2-r_2}\frac{r_1r_2+r_1j+r_2i}{n_1n_2}\proba{X_{n_1}^{1,1}=-r_1-i}\proba{X_{n_1}^{1,2}=j}\proba{X_{n_2}^{2,1}=i}\proba{X_{n_2}^{2,2}=-r_2-j}.
\end{split}
\end{equation}We perform each summation in \emph{columns}, obtaining three terms of the form
\begin{equation}\label{eqnExampleChangeSummationOverColumns}
	\frac{1}{n_1n_2}\sum_{i=0}^{n_1-r_1}k_{\ell_1,1}\proba{X_{n_1}^{1,1}=-r_1-i}\proba{X_{n_2}^{2,1}=i}
	\sum_{j=0}^{n_2-r_2}k_{\ell_2,2}\proba{X_{n_1}^{1,2}=j}\proba{X_{n_2}^{2,2}=-r_2-j},
\end{equation}where $k_{\ell_1,1}\in \{r_1,i\}$ and $k_{\ell_2,2}\in \{r_2,j\}$. 
Each of the above sums can be written as convolutions, either multiplied by a constant or a random variable (depending on the value of $k_{i_j,j}$).
Is precisely the second case where we use Hypotheses \hyperref[descriptionCMGConvolution]{\bo{H2}}. 
First, we describe explicitly $det(-\bb k)$. 

\begin{defi}\label{defiElementaryForest}
	An \emph{elementary forest} is a multitype forest that contains exactly one vertex of each type. 
	In particular, each elementary forest contains exactly $d$ vertices and is coded by the $d$ couples $(\mathfrak{e}_j,j)$ for $j\in [d]$, where $\mathfrak{e}_j$ is the type of the parent of vertex type $j$. 
	If the vertex of type $j$ is a root, then we set $\mathfrak{e}_j=0$.
	We define the set $\bb F_d^{\text{elem.}}$ of vectors $\mathfrak{e}=(\mathfrak{e}_1,\ldots, \mathfrak{e}_d)$, with $0\leq \mathfrak{e}_j\leq d$ such that $(\mathfrak{e}_j,j),\ i\in [d]$ codes an elementary forest.
\end{defi}

Recall Definition \ref{defiCodifyingSequencesOfMultitypeForestsWithR} of $S^{\bo r}_d$ of coding sequences of multitype forests.
Define $\bb M(\z)$ as the set of matrices with $d\times d$ entries, all in $\z$. 
In the following, 
we write $k_{\mathfrak{e}_j,j}$ instead of $k_{\mathfrak{e}_jj}$.

\begin{lemma}[Lemma 4.5 in \cite{MR3449255}]\label{lemmaNumberOfGoodCyclicalPerm}
Let $\bo r\in \z^d_+$ and $\bb k=(k_{i,j})\in \bb M(\z)$ satisfying $r_j+ (\bo 1\cdot \bb k)_j=0$ for every $j\in [d]$. Then, defining $k_{0,j}:=r_j$ we have
\begin{equation*}
	\det(-\bb k)=\sum_{\mathfrak{e}\in \bb F_d^{\text{elem.}}}\prod_{j=1}^{d}k_{\mathfrak{e}_j,j}.
\end{equation*}
\end{lemma}

\begin{proof}[Proof of Theorem \ref{teoLawOfTheTotalPopulationByTypes}]\label{proofteoLawOfTheTotalPopulationByTypes}
Define the index sets
\begin{equation*}
	I(\bo{r,n})=\{ \bb k=(k_{i,j})_{i,j\in[d]}:k_{i,j}\geq 0\mbox{ for }i\neq j,0\leq -k_{j,j}\leq n_j,0=r_j+(\bo 1\cdot \bb k)_j,\forall j\in [d] \},
\end{equation*}and
\begin{equation*}
	I'(\bo{r,n})=\{ (k_{i,j})_{i,j\in [d],\ i\neq j}:k_{i,j}\geq 0\mbox{ for }i\neq j,0\leq r_j+\sum_{i\neq j}k_{i,j}\leq n_j,\forall j\in [d] \},
\end{equation*}and use the notation $\sum_{\bb k\in I(\bo{r,n})}$ (resp. $\sum_{(k_{i,j})\in I'(\bo{r,n})}$) to denote the summation over all matrices $\bb k=(k_{i,j})$ with $i,j\in [d]$, such that $\bb k\in I(\bo{r,n})$ (resp. all values $(k_{i,j})$ with $i\neq j$ such that $(k_{i,j})\in I'(\bo{r,n})$).

It is clear that
\[
\bo P\paren{\bo O_{\bo{r}}=\bo n} =\sum_{(k_{i,j})\in I'(\bo{r,n})}\bo P\paren{\bo O_{\bo{r}}=\bo n,A_{i,j}=k_{i,j},\ i,j\in[d],i\neq j}.
\]Indeed, by definition of $(A_{i,j})$, we have $A_{i,j}\geq 0$, and when $n_j<r_j+\sum_{i\neq j}k_{i,j}$, then there are at least $r_j+\sum_{i\neq j}k_{i,j}$ individuals type $j$ which contradicts $O_j=n_j$.

By the independence imposed on $\mu_j$ (Hypothesis \hyperref[descriptionCMGIndependence]{\bo{H1}}), the product in Equation \eqref{eqnPopulationByTypesAndParents} separates into
\begin{align*}
	P(\bb k)&:=\prod_1^d\mu_j^{*n_j}\paren{k_{j1},\ldots ,k_{j(j-1)},n_j+k_{j,j},k_{j(j+1)},\ldots, k_{jd}}\\
	&=\prod_{j=1}^d \proba{X^{j,j}_{n_j}= -r_j-\sum_{i\neq j}k_{i,j}}\prod_{i\neq j}\proba{X^{i,j}_{n_i}=k_{i,j}}\\
	&=\prod_{j=1}^d \mu^{*n_j}_{j,j}\paren{-r_j-\sum_{i\neq j}k_{i,j}}\prod_{i\neq j}\mu^{*n_i}_{i,j}\paren{k_{i,j}},
\end{align*}where $\mu^{*n_i}_{i,j}$ is the law of $X^{i,j}_{n_i}$.

Hence, using Theorem \ref{teoChaumontTotalPopulationByTypesandParents} for each $\bb k\in I(\bo{r},\bo{n})$, we need to compute
\begin{align*}
\bo P\paren{\bo O_{\bo{r}}=\bo n}& =\frac{1}{\prod n_i}\sum_{\bb k\in I(\bo{r,n})}det(-\bb k)P(\bb k).
\end{align*}
Since on the summation we consider $\bb k'$s with $0=r_j+(\bo 1\cdot \bb k)_{j}$, then the conditions of Lemma \ref{lemmaNumberOfGoodCyclicalPerm} are satisfied. 
Thus using $k_{0j}=r_j$, we have
\begin{align*}
	\bo P\paren{\bo O_{\bo{r}}=\bo n}&	=\frac{1}{\prod n_i}\sum_{\mathfrak{e}\in \bb F_d^{\text{elem.}}} \sum_{\bb k\in I(\bo{r,n})}P(\bb k)\prod_{j=1}^dk_{\mathfrak{e}_j,j} \\
	& =\frac{1}{\prod n_i}\sum_{\mathfrak{e}\in \bb F_d^{\text{elem.}}}\sum_{(k_{i,j})\in I'(\bo{r,n})}\prod_{j=1}^dk_{\mathfrak{e}_j,j} \mu^{*n_j}_{j,j}(-r_j-\sum_{i\neq j}k_{i,j})\prod_{i\neq j}\mu^{*n_i}_{i,j}(k_{i,j}).
\end{align*}We perform the summation in columns, that is, for every $j$, we join all the terms depending only on $(k_{i,j},i\in [d],i\neq j)$, as explained in \eqref{eqnExampleChangeSummationOverColumns}. 
That is, fixing $\mathfrak{e}\in \bb F_d^{\text{elem.}}$, the sum over all values $(k_{i,j})\in I'(\bo r,\bo n)$ is the same as the sum over all values $(k_{i1})_{i\in [d],i\neq 1}$ of $k_{\mathfrak{e}_1,1} \mu^{*n_1}_{1,1}(-r_1-\sum_{i\neq 1}k_{i1})\prod_{i\neq 1}\mu^{*n_i}_{i,1}(k_{i1})$, times the sum over all values $(k_{i2})_{i\in [d],i\neq 2}$ of $k_{\mathfrak{e}_2,2} \mu^{*n_2}_{2,2}(-r_2-\sum_{i\neq 2}k_{i2})\prod_{i\neq 2}\mu^{*n_i}_{i,2}(k_{i2})$, and so on. Hence, if for every $j$ we define 
\begin{equation*}
	I^j(r_j,n_j)=\{ (k_{i,j})_{i\in [d],\ i\neq j}:k_{i,j}\geq 0\mbox{ for }i\neq j,0\leq r_j+\sum_{i\neq j}k_{i,j}\leq n_j\},
\end{equation*}then
\begin{align*}
	\bo P\paren{\bo O_{\bo{r}}=\bo n}
	& =\frac{1}{\prod n_i}\sum_{\mathfrak{e}\in \bb F_d^{\text{elem.}}}\prod_{j=1}^d\paren{\sum_{(k_{i,j})_{i\in [d],i\neq j}\in I^j(r_j,n_j)}k_{\mathfrak{e}_j,j} \mu^{*n_j}_{j,j}(-r_j-\sum_{i\neq j}k_{i,j})\prod_{i\neq j}\mu^{*n_i}_{i,j}(k_{i,j})}\\
	& =\frac{1}{\prod n_i}\sum_{\mathfrak{e}\in \bb F_d^{\text{elem.}}}\prod_{j=1}^d\esp{X^{\mathfrak{e}_j,j}_{n_{\mathfrak{e}_j}} \mu^{*n_j}_{j,j}\paren{-r_j-\sum_{i\neq j}X^{i,j}_{n_{i}}}},
\end{align*}using the convolution formula. 
Note that whenever $r_j=n_j$ for some $j$, the above formula still is true since $	I^j(r_j,n_j)$ contains only the null vector, and hence the sum over the latter equals $\mu^{*n_j}_{j,j}(0)\prod_{i\neq j} \mu^{*n_i}_{i,j}(0)=\esp{\mu^{*n_j}_{j,j}\paren{-r_j-\sum_{i\neq j}X^{i,j}_{n_{i}}}}$. 

Fix distinct $i,j\in[d]$. 
Let $\mathcal{F}_{-j} = \sigma(X^{\ell,j}_{n_\ell} , \ell\in [d]\setminus \{j\})$ be the $\sigma$-algebra generated by all random walks except the $j$-th one. 
Conditioning on $\mathcal{F}_{-j}$ and by the independence assumption \hyperref[descriptionCMGIndependence]{\bo{H1}}, we obtain 
\begin{align*}
	\esp{X^{i,j}_{n_i} \indi{\sum_{\ell=1}^d X^{\ell,j}_{n_{\ell}} = -r_j} \Biggm| \mathcal{F}_{-j}}
	&= X^{i,j}_{n_i} \esp{\indi{X^{j,j}_{n_j} = -r_j - \sum_{\ell\neq j} X^{\ell,j}_{n_{\ell}}} \Biggm| \mathcal{F}_{-j}} \\
	&= X^{i,j}_{n_i} \mu^{*n_j}_{j,j}\paren{-r_j-\sum_{\ell\neq j}X^{\ell,j}_{n_{\ell}}}.
\end{align*}
Taking the expectation on both sides and applying the tower property, we have
\begin{align*}
\prod_{j=1}^d\esp{X^{\mathfrak{e}_j,j}_{n_{\mathfrak{e}_j}} \mu^{*n_j}_{j,j}\paren{-r_j-\sum_{\ell\neq j}X^{\ell,j}_{n_{\ell}}}} 
&= \prod_{j=1}^d \esp{X^{\mathfrak{e}_j,j}_{n_{\mathfrak{e}_j}} \indi{\sum_{\ell=1}^d X^{\ell,j}_{n_{\ell}} = -r_j}} \\
&= \prod_{j=1}^d \widehat{k}_{\mathfrak{e}_j,j}(\bo{n}, \bo{r}) \proba{\sum_{\ell=1}^d X^{\ell,j}_{n_\ell} = -r_j},
\end{align*}
where $\widehat{k}_{i,j}(\bo{n}, \bo{r}) = \esp{X^{i,j}_{n_i} \mid \sum_{\ell=1}^d X^{\ell,j}_{n_\ell} = -r_j}$ for $i \neq 0$, and $\widehat{k}_{0,j}(\bo{n}, \bo{r}) = r_j$. By Lemma 4.5 of \cite{MR3449255}, summing the product $\prod_{j=1}^d \widehat{k}_{\mathfrak{e}_j,j}(\bo{n}, \bo{r})$ over all elementary forests $\mathfrak{e} \in \bb F_d^{\text{elem.}}$ is exactly equal to $\det(-\widehat{\bb k}(\bo r, \bo n))$, establishing \eqref{eqnGeneralDeterminantFormula}.

Now we prove the particular expressions derived under different hypothesis.

\textit{Assuming \hyperref[descriptionCMGConvolution]{\bo{H2}}}.
Note that $\widehat{k}_{\mathfrak{e}_j,j}(\bo{n}, \bo{r})$ is precisely the left-hand side of Hypotheses \hyperref[descriptionCMGConvolution]{\bo{H2}}. 
	Thus, from the latter 
	\[\widehat{k}_{\mathfrak{e}_j,j}(\bo{n}, \bo{r})= \begin{cases} 
	r_j & \mathfrak{e}_j=0 \\
	n_{\mathfrak{e}_j}(n_j-r_j)/\#\bo n & \mathfrak{e}_j\neq 0.
	\end{cases}
	\]Using Lemma \ref{lemmaDeterminantUnderHypothesis} in the Appendix, this implies
\[
\frac{1}{\prod n_i}\sum_{\mathfrak{e}\in \bb F_d^{\text{elem.}}}\prod_{j=1}^{d}\widehat{k}_{\mathfrak{e}_j,j}(\bo{n}, \bo{r})=\frac{\#\bo r}{\#\bo n}.
\]

\textit{Assuming \hyperref[descriptionCMGProportional]{\bo{H4}}.}
We define the conditional expectations $\widehat{k}_{i,j}(\bo{n}, \bo{r}) = u_i v_j$ for $i \neq j$, where $u_i = \frac{w_i n_i}{\sum_{\ell=1}^d w_\ell n_\ell}$ and $v_j = n_j - r_j$. The diagonal elements of $-\widehat{\bb k}(\bo r,\bo n)$ satisfy:
\begin{align*}
    -\widehat{k}_{j,j}(\bo{n}, \bo{r}) &= r_j + \sum_{\ell \neq j} u_\ell v_j = r_j + v_j(1 - u_j) = r_j + (n_j - r_j) - u_j v_j = n_j - u_j v_j.
\end{align*}
Therefore, we can decompose $-\widehat{\bb k} = D - \bo u \bo v^T$, where $D = \text{diag}(n_1, \dots, n_d)$. Applying the Matrix Determinant Lemma:
\begin{align*}
    \det(-\widehat{\bb k}) &= \det(D)(1 - \bo v^T D^{-1} \bo u) 
    = \left(\prod_{i=1}^d n_i\right) \left( 1 - \sum_{j=1}^d \frac{n_j - r_j}{n_j} \frac{w_j n_j}{\sum_{\ell=1}^d w_\ell n_\ell} \right) \\
    &= \left(\prod_{i=1}^d n_i\right) \frac{\sum_{\ell=1}^d w_\ell n_\ell - \sum_{j=1}^d w_j(n_j - r_j)}{\sum_{\ell=1}^d w_\ell n_\ell} 
    = \left(\prod_{i=1}^d n_i\right) \frac{\sum_{\ell=1}^d w_\ell r_\ell}{\sum_{\ell=1}^d w_\ell n_\ell}.
\end{align*}
Dividing this determinant by $\prod n_i$ directly yields the factor $\frac{\sum_{i=1}^d w_i r_i}{\sum_{\ell=1}^d w_\ell n_\ell}$, proving \eqref{eqnRankOneSimplification}.

\textit{Assuming \hyperref[descriptionCMGProgressive]{\bo{H5}}. } Because $X^{i,j}_1 = 0$ almost surely for $i > j$, we have $\widehat{k}_{i,j}(\bo{n}, \bo{r}) = 0$ for $i > j$. Consequently, the matrix $-\widehat{\bb k}$ is upper triangular. The determinant of an upper triangular matrix is exactly the product of its diagonal entries. Since $-\widehat{k}_{j,j}(\bo{n}, \bo{r}) = r_j + \sum_{i \neq j} \widehat{k}_{i,j}(\bo{n}, \bo{r}) = r_j + \sum_{i < j} \widehat{k}_{i,j}(\bo{n}, \bo{r})$, we obtain:
\begin{equation*}
    \det(-\widehat{\bb k}) = \prod_{j=1}^d \left( r_j + \sum_{i < j} \widehat{k}_{i,j}(\bo{n}, \bo{r}) \right).
\end{equation*}
This directly yields Equation \eqref{eqnDAGSimplification}, completing the proof.
\end{proof}

We now prove Corollary  \ref{coroRelatingFirstHittingTimeAndBridges}. 
This result will be used when computing the complexity time of our algorithm 

\begin{proof}[Proof of Corollary \ref{coroRelatingFirstHittingTimeAndBridges}]
Let $\mathbb{E}_{\bo n, \bo r} \subset \mathbb{B}_{\bo n, \bo r}$ be the subset of multidimensional bridges that are strictly minimal solutions (i.e., paths $\bb w$ that successfully code a valid forest, satisfying $\bo T(\bo r,\bb w) = \bo n$). By the Multivariate Cyclic Lemma \cite[Lemma 3.3]{MR3449255}, which applies by the assumption on $\bo r,\bo n$ and $\bb X$, for any bridge $\bb x \in \mathbb{B}_{\bo n, \bo r}$, exactly $\det(-\bb k(\bb x))$ of its $\prod_{i=1}^d n_i$ cyclical permutations $\theta_{\bo q, \bo n}(\bb x)$ belong to $\mathbb{E}_{\bo n, \bo r}$. We can express this counting fact as 
\begin{equation*}
    \det(-\bb k(\bb x)) = \sum_{\bo 0 \le \bo q \le \bo n - \bo 1_d} \indi{\theta_{\bo q, \bo n}(\bb x) \in \mathbb{E}_{\bo n, \bo r}}.
\end{equation*}Since $\bb X$ and $F_{\bo n}$ are invariant under cyclical permutations, and there are $\prod n_i$ of them
\begin{align*}
\E\paren{F_{\bo n}(\bb X);\bo T(\bo r,\bb X)=\bo n}&= \E\paren{F_{\bo n}(\theta_{\bo q, \bo n}(\bb X));\bo T(\bo r,\theta_{\bo q, \bo n}(\bb X))=\bo n}\\
&= \frac{1}{\prod n_i}\sum_{\bo 0 \le \bo q \le \bo n - \bo 1_d}\E\paren{F_{\bo n}(\theta_{\bo q, \bo n}(\bb X));\theta_{\bo q, \bo n}(\bb X) \in \mathbb{E}_{\bo n, \bo r}}\\
&= \frac{1}{\prod n_i}\sum_{\bo 0 \le \bo q \le \bo n - \bo 1_d}\E\paren{F_{\bo n}(\bb X);\theta_{\bo q, \bo n}(\bb X) \in \mathbb{E}_{\bo n, \bo r}}\\
&= \frac{1}{\prod n_i}\E\paren{F_{\bo n}(\bb X)\indi{\bb X\in \mathbb{B}_{\bo n, \bo r}}\sum_{\bo 0 \le \bo q \le \bo n - \bo 1_d}\indi{\theta_{\bo q, \bo n}(\bb X) \in \mathbb{E}_{\bo n, \bo r}}}\\
&=\E\paren{F_{\bo n}(\bb X)\frac{{\rm det}(-\bb k(\bb X))}{\prod n_i};\bb X\in \bb B_{\bo n,\bo r}}.
\end{align*}The last assertion of the lemma follows directly from Theorem \ref{teoLawOfTheTotalPopulationByTypes}. 
\end{proof}

Now we obtain the total size of the MBGW tree, using the result in Theorem \ref{teoLawOfTheTotalPopulationByTypes}. 

\begin{proof}[Proof of Corollary \ref{coroLawOfTheTotalPopulationInAMBGW}]
	By the law of total probability, we sum over all valid configurations of $n_1, \dots, n_d$ such that $\sum n_\ell = n$ and $n_j \ge r_j$ for all $j\in [d]$. Because $\bo r \in \mathbb{N}^d$, we safely guarantee that $n_j \ge r_j \ge 1$, which is required in Theorem \ref{teoLawOfTheTotalPopulationByTypes}. Using Theorem \ref{teoLawOfTheTotalPopulationByTypes} under \hyperref[descriptionCMGConvolution]{\bo{H2}}, we have:
	\begin{align*}
	\bo P\paren{\#\bo O_{\bo r}=n} &= \sum_{\substack{n_j,\ge r_j \,\forall\,j\\ \sum n_j = n}} \bo P\paren{\bo O_{\bo r} = (n_1,\ldots, n_d)} \\
	&= \frac{\#\bo r}{n} \sum_{\substack{n_j,\ge r_j \,\forall\,j \\ \sum n_j = n}} \prod_{j=1}^d \proba{\sum_{\ell=1}^d X^{\ell,j}_{n_\ell} = -r_j}.
	\end{align*}
	Let $k_j = n_j - r_j$. The condition $\sum_{j=1}^d n_j = n$ becomes $\sum_{j=1}^d k_j = n - \#\bo r$. Thus, we obtain:
	\begin{align*}
	\bo P\paren{\#\bo O_{\bo r}=n} &= \frac{\#\bo r}{n} \sum_{\substack{k_1, \dots, k_d \ge 0 \\ \sum k_j = n-\#\bo r}} \prod_{j=1}^d f_n^{(j)}(k_j).
	\end{align*}
	Since $f_n^{(j)}$ represents the probability mass function of the independent random variables $Y_n^{(j)}$, the sum over all valid non-negative partitions $k_1 + \dots + k_d = n - \#\bo r$ is exactly the discrete convolution evaluated at $n-\#\bo r$. Thus, the summation equals $\proba{\sum_{j=1}^d Y^{(j)}_{n} = n - \#\bo r}$.
\end{proof}

We now show that, whenever the offspring distribution is \emph{parent-independent}, Hypothesis \hyperref[descriptionCMGConvolution]{\bo{H2}} is automatically satisfied by any discrete probability distribution because of exchangeability.
We will write $\mu_i=(\mu_{i,1},\ldots, \mu_{i,d})$. 

\begin{lemma}\label{lemmaExchangeabilityH2}
Assume \hyperref[descriptionCMGIndependence]{\bo{H1}} holds. If for every $j \in [d]$, the marginal distribution of the number of type $j$ children is independent of the parent's type $i$ (that is, $\mu_{i,j} \equiv f_j$ for all $i \in [d]$), then Hypothesis \hyperref[descriptionCMGConvolution]{\bo{H2}} is automatically satisfied.
\end{lemma}
\begin{proof}
Recall that $\kappa_j(u_{m}^{(\ell)})$ is the number of type $j$ children born to the $m$-th individual of type $\ell$. 
By \hyperref[descriptionCMGIndependence]{\bo{H1}}, for a fixed $j$, the random variables $\big( \kappa_j(u_{m}^{(\ell)});\ell \in [d], 1 \le m \le n_\ell\big)$ form a sequence of $\#\bo n = \sum_{\ell=1}^d n_\ell$ i.i.d. random variables.

Let $S_j = \sum_{\ell=1}^d \sum_{m=1}^{n_\ell} \kappa_j(u_{m}^{(\ell)})$ be the total number of type $j$ children in the forest. By the exchangeability of i.i.d. random variables, the conditional expectation of any single variable given the total sum $S_j = n_j-r_j$ is simply the average $(n_j-r_j)/\#\bo n$. 
It follows that
\begin{equation*}
	\esp{X^{i,j}_{n_i} \;\middle|\; S_j = n_j - r_j} = \sum_{m=1}^{n_i} \esp{\kappa_j(u_{m}^{(i)}) \;\middle|\; S_j = n_j - r_j} = \frac{n_i}{\#\bo n} (n_j - r_j),
\end{equation*}which is exactly  \hyperref[descriptionCMGConvolution]{\bo{H2}}.
\end{proof}

\subsection{Application to the enumeration of plane, labeled and binary multitype forests with given roots and types sizes}\label{subsectionExamplesMultitypeGeoPoiBer}

In this section we prove Propositions \ref{propoH2Geometric}, \ref{propoH2Poisson} and \ref{propoH2Binomial}, providing three examples where the Hypotheses \hyperref[descriptionCMGConvolution]{\bo{H2}} are satisfied, under the assumptions of Theorem \ref{teoLawOfTheTotalPopulationByTypes}.
Noe that in the three examples, Hypothesis \hyperref[descriptionCMGConvolution]{\bo{H2}} is satisfied by Lemma \ref{lemmaExchangeabilityH2}. 
We consider MBGW forests with mean matrix $M$ having $i$-th row of the form $(m_{i1},\ldots, m_{id})=(p_1,\ldots, p_d)$ with $p_j\geq 0$ and $\sum_{j\in [d]}p_j=1$. 
In this case, the process is critical since trivially $M$ has right-eigenvector $(1,\ldots, 1)$. 
For the geometric distribution, this is accomplished assuming $\sum _{j\in [d]}(1-p_j)/p_j\leq 1$.
For the Poisson we need $\sum \mu_j\leq 1$ and for the Bernoulli $2\sum p_j\leq 1$ since individuals have zero or two children.  

\begin{remark}
The conditions on the parameters are imposed to work with (sub)critical MBGW forests (c.f. Section 2.3 of \cite{MR2908619}, where for single-type simply generated trees, the parameters can be changed by working with equivalent weights). 
\end{remark}

\subsubsection{Geometric Offspring}
Recall that for any $\bo f\in \mathbb{F}^{plane}_{\bo{n,r}}$, we denote by $\bo f^{(i)}$ the subforest of type $i$ of $\bo f$. 
Denoting by $c_i(v)$ the number of children type $i$ that vertex $v$ has, then 
\begin{align*}
\bo P\paren{\mathcal{G}_{{\bf r},\bo p}=\bo f}
& = \prod_{j\in [d]}\prod_{v\in \bo f^{(j)}}\mu_j(c_1(v),\ldots, c_d(v))\\
& = \prod_{v\in \bo f}\prod_{j\in [d]}p_j(1-p_j)^{c_j(v)}\\
& = \prod_j p_j^{\#\bo n} \prod_j (1-p_j)^{\mbox{number of children type j}}\\
&= \prod_j p_j^{\#\bo n} \prod_j (1-p_j)^{n_j-r_j},
\end{align*}where the product in the second line is over all $\#\bo n$ vertices in $\bo f$.

Thus, using Theorem \ref{teoLawOfTheTotalPopulationByTypes}, we obtain
\begin{equation*}
	\bo P\paren{\#\mathcal{G}_{{\bf r},\bo p}=\bo n}=\frac{\#\bo r}{\#\bo n}\prod_j {\#\bo n+n_j-r_j-1\choose n_j-r_j}p_j^{\#\bo n}(1-p_j)^{n_j-r_j}.
\end{equation*}Joining the above computations, we obtain Proposition \ref{propoH2Geometric}. 
Note that we also obtain the distributional equality
\begin{equation*}
	\mathcal{F}^{plane}_{\bo{n,r}}\stackrel{d}{=}\paren{\mathcal{G}_{{\bf r},\bo p}|\#\mathcal{G}_{{\bf r},\bo p}=\bo n},
\end{equation*}where $\mathcal{F}^{plane}_{\bo{n,r}}$ is uniform on $\mathbb{F}^{plane}_{\bo{n,r}}$.
Taking $d=1$ agrees with the unidimensional case, see Formula (35) in \cite{MR1630413}.

\subsubsection{Poisson Offspring}
For simplicity consider $d=2$. 
Fix  any $\bo f\in \mathbb{F}^{plane}_{\bo{n,r}}$ having $r_i$ roots type $i$, $k_{1,2}$ type 2 individuals whose parent is of type 1, and $k_{2,1}$ type 1 individuals whose parent is of type 2. 
Then
\begin{align*}
 \bo P\paren{\mathcal{P}_{{\bf r},\bm \mu}=\bo f}
& = e^{-\mu_1 \#\bo n}\mu_1 ^{n_1-r_1-k_{21}} \mu_2 ^{k_{12}}\prod_{v\in \bo f^{(1)}}\frac{1}{c_1(v)!}\prod_{v\in \bo f^{(1)}}\frac{1}{c_2(v)!}\\\
&\quad \times e^{-\mu_2 \#\bo n}\mu_2 ^{n_2-r_2-k_{12}} \mu_1 ^{k_{21}}\prod_{v\in \bo f^{(2)}}\frac{1}{c_1(v)!}\prod_{v\in \bo f^{(2)}}\frac{1}{c_2(v)!}\\
& = e^{-(\mu_1+\mu_2)\#\bo n }\mu_1 ^{n_1-r_1}\mu_2 ^{n_2-r_2} \prod_{v\in \bo{f}}\frac{1}{c_1(v)!c_2(v)!}.
\end{align*}

Using Theorem \ref{teoLawOfTheTotalPopulationByTypes} we obtain
\begin{equation*}
	\bo P\paren{\#\mathcal{P}_{{\bf r},\bm \mu}=\bo n}
	= \frac{\#\bo r}{\#\bo n}\prod_j\frac{e^{-\#\bo n\mu_j}(\#\bo n\mu_j)^{n_j-r_j}}{(n_j-r_j)!}.
\end{equation*}A similar formula has been obtained in Lemma 4.7 of \cite{MR4586227}, in the setting of inhomogeneous random graphs. It follows that
\begin{equation*}
	\bo P\paren{\mathcal{P}_{{\bf r},\bm \mu}=\bo f \mid \#\mathcal{P}_{{\bf r},\bm \mu}=\bo n}=\frac{\prod_j{n_j-r_j \choose c_j(v_1),\ldots ,c_j(v_{\#\bo n})}}{\frac{\#\bo r}{\#\bo n}\#\bo n^{\#\bo n-\#\bo r}}\ \ \ \ \ \ \ \forall \ \bo f\in \mathbb{F}^{plane}_{\bo{n,r}}.
\end{equation*}Note that for $d=1$ this agrees with the unidimensional case, as seen in Formula (39) of \cite{MR1630413}. 
Since the right-hand side depends on $\bo f$, it is not uniform on the set of plane forests. 
To obtain a uniform forest, we introduce a function as in \cite{MR1630413}.
Define $\Psi:\mathbb{F}^{labeled}_{\bo{n,r}}\mapsto \mathbb{F}^{plane}_{\bo{n,r}}$ as follows:
\begin{enumerate}
	\item Order the trees of the forest, according to the natural order of the labels in the roots of type 1, then order the type 2 roots, and so on.

	\item For each vertex $v_i$ of type $i$, order its $c_1(v_i)$ children of type 1 according to its labels, its $c_2(v_i)$ children of type 2 according to its labels, and so on.
	\item Erase the labels.
\end{enumerate}Now, we find the number of forests in $\mathbb{F}^{labeled}_{\bo{n,r}}$ that are sent to a given plane forest $\bo f$. 
For each $i$, there are $(n_i-r_i)!$ ways to label the type $i$ vertices (recall that our rooted labeled forests have root set $[r]$).
But the permutation of the children of a fixed type of each vertex also lead to the same forest $\bo f$.
That is, if vertex $v$ has $c_i(v)$ children type $i$, there are $c_i(v)!$ labelings of such children leading to $\bo f$. 
This being true for every type and every vertex, we have
\begin{equation*}
	\# \Psi^{-1}(\bo f)=\frac{\prod_j (n_j-r_j)!}{\prod_{v\in \bo f}\prod_j c_j(v)!}.
\end{equation*}This is exactly the numerator in the formula obtained above.
Thus, we have the following interpretation: let $\mathcal{F}^{labeled}_{\bo{n,r}}$ have uniform distribution over the set of all $d$-type labeled forests, where the roots are in $[r]$, with roots-type $\bo{r}$ and individuals-type $\bo{n}$, and let $\mathcal{P}^*_{\bo{n,r}}$ be $\mathcal{P}_{\bo{n,r}}$ relabeled by $d$ uniform random permutations, one for each type, then
\begin{equation*}
	\mathcal{F}^{labeled}_{\bo{n,r}}\stackrel{d}{=}\paren{\mathcal{P}^*_{{\bf r},\bm \mu}\left|\#\mathcal{P}_{{\bf r},\bm \mu}=\bo n\right.}.
\end{equation*}We note that the previous formulas coincide with the results in \cite[Section 7]{MR1630413} for the unitype case.
This formula also provides a direct structural connection between multitype enumerations and classical unitype labeled forests. By Cayley's generalized formula for forests, the number of unitype labeled forests on a vertex set of size $\#\bo n$ having exactly $\#\bo r$ specified roots is precisely $\frac{\#\bo r}{\#\bo n}(\#\bo n)^{\#\bo n-\#\bo r}$. We can justify this numerical equivalence bijectively: any multitype labeled forest $\bo f \in \mathbb{F}^{labeled}_{\bo{n,r}}$ uniquely corresponds to a standard (unitype) labeled forest on $[\#\bo n]$ with root set $[\#\bo r]$. The $\#\bo r$ roots naturally retain their labels in $[\#\bo r]$. The $n_1-r_1$ non-root individuals of type $1$ are assigned the consecutive labels from $\#\bo r+1$ to $\#\bo r+n_1-r_1$ preserving their original relative order; the $n_2-r_2$ non-root individuals of type $2$ take the subsequent $n_2-r_2$ available labels, and so on for all types. Since the multitype structure simply partitions the labels into prescribed sizes for each type without affecting the tree topologies, this relabeling is a strict bijection. Thus, the multitype restriction perfectly preserves the exact total count of unitype labeled forests.

\subsubsection{Bernoulli Offspring}

In this case, since any vertex $v$ has zero or two children with probability $p$ or $1-p$ respectively, then $\mu_i(c_1(v),c_2(v))=p_1^{c_1(v)/2}(1-p_1)^{1-c_1(v)/2}p_2^{c_2(v)/2}(1-p_2)^{1-c_2(v)/2}$ for $d=2$.

As before, consider any $\bo f\in \mathbb{F}^{binary}_{\bo{n,r}}$ 
Note that 
$n_i-r_i$ for $i=1,2$
are even numbers.
Hence
\begin{align*}
\bo P\paren{\mathcal{B}_{{\bf r},\bo p}=\bo f}
& = \prod_v\prod_j \paren{\frac{p_j}{1-p_j}}^{\mbox{$\frac{1}{2}$\# children type $j$ of v}}(1-p_j)
&= \prod_j \paren{\frac{p_j}{1-p_j}}^{(n_j-r_j)/2}(1-p_j)^{\#\bo n}.
\end{align*}

From Theorem \ref{teoLawOfTheTotalPopulationByTypes} we obtain 
\begin{equation*}
	\bo P\paren{\#\mathcal{B}_{{\bf r},\bo p}=\bo n} = \frac{\#\bo r}{\#\bo n}\prod_j{ \#\bo n \choose (n_j-r_j)/2 } p_j^{(n_j-r_j)/2}(1-p_j)^{\#\bo n-(n_j-r_j)/2}.
\end{equation*}Putting all pieces together, we have shown Proposition  \ref{propoH2Binomial}.
Just as before, we also obtain 
\begin{equation*}
	\mathcal{F}^{binary}_{\bo{n,r}}\stackrel{d}{=}\paren{\mathcal{B}_{{\bf r},\bo p}|\#\mathcal{B}_{{\bf r},\bo p}=\bo n},
\end{equation*}where $\mathcal{F}^{binary}_{\bo{n,r}}$ is uniform on $\mathbb{F}^{binary}_{\bo{n,r}}$.
Compare this formula with the number of binary trees, which is related to the Catalan numbers, see Theorem 2.1 in \cite{MR2484382}.

\begin{propo}\label{propoH2NegativeBinomial}
	Assume the offspring distribution is given by independent Negative Binomial distributions, where $\mu_i(k_1,\dots,k_d) = \prod_{j\in[d]} \binom{k_j + m_j - 1}{k_j} p_j^{m_j} (1-p_j)^{k_j}$ for fixed integers $m_j \ge 1$ and $p_j \in (0,1)$. Let $\sum_{j\in[d]} m_j \frac{1-p_j}{p_j} \le 1$. Let $\mathcal{NB}_{{\bf r},\bo p}$ be a multitype forest with such offspring distribution. Then
	\begin{equation*}
		\bo P\paren{\mathcal{NB}_{{\bf r},\bo p}=\bo f|\#\mathcal{NB}_{{\bf r},\bo p}=\bo n}=\frac{\prod_{v \in \bo f}\prod_{j\in[d]} \binom{c_j(v) + m_j - 1}{c_j(v)}}{\frac{\#\bo r}{\#\bo n} \prod_{j\in[d]} \binom{n_j - r_j + \#\bo n m_j - 1}{n_j - r_j}} \ \ \ \ \ \ \forall\,\bo f\in \mathbb{F}^{plane}_{\bo{n,r}}.
	\end{equation*}
\end{propo}
For $m_j = 1$, the combinatorial products in the numerator collapse to $1$, exactly recovering the uniform plane forest distribution of Proposition \ref{propoH2Geometric}.

\begin{propo}\label{propoH2BinomialGeneral}
	Assume the offspring distribution is given by independent Binomials, that is $\mu_i(k_1,\dots,k_d) = \prod_{j\in[d]} \binom{m_j}{k_j} p_j^{k_j} (1-p_j)^{m_j-k_j}$ for fixed integers $m_j \ge 1$ and $p_j \in (0,1)$. Let $\sum_{j\in[d]} m_j p_j \le 1$. Let $\mathcal{B}'_{{\bf r},\bo p}$ be a multitype forest with such offspring distribution. Then
	\begin{equation*}
		\bo P\paren{\mathcal{B}'_{{\bf r},\bo p}=\bo f|\#\mathcal{B}'_{{\bf r},\bo p}=\bo n}=\frac{\prod_{v \in \bo f}\prod_{j\in[d]} \binom{m_j}{c_j(v)}}{\frac{\#\bo r}{\#\bo n}\prod_{j\in[d]} {\#\bo n m_j \choose n_j-r_j}} \ \ \ \ \ \ \forall\,\bo f\in \mathbb{F}^{plane}_{\bo{n,r}}.
	\end{equation*}
\end{propo}
Here, the numerator evaluates the number of ways to embed $\mathbf{f}$ into a strict $\bo m$-ary structure. Thus, this distribution enumerates  multitype $\bo m$-ary forests.

\section{Simulation of CMBGW forests}\label{sectionAlgorithmsConditionedRandomForests}

In this section, we utilize Theorem \ref{teoConstructionOFUniformMtypeForestWithGDS} to construct a MFGDS where the degree sequence itself is random. By combining this construction with the joint law of the total population by types established in Theorem \ref{teoLawOfTheTotalPopulationByTypes}, we prove our main result, Theorem \ref{teoSimulationOfCMBGWIntro}. We begin by demonstrating that, under an independence condition, a MBGW forest conditioned on its multitype degree sequence is equal in law to a MFGDS.

\subsection{Relation between MFGDS and CMBGW forests}\label{secRelationMFGDScMBGW}

For any given $\bo{n}=(n_1,\ldots, n_d)\in \mathbb{N}^d$ and $\bo{r}=(r_1,\ldots,r_d)\in \z_+^d$ with $\sum r_i>0$ and $\bo r\leq \bo n$, define the set of all Degree Sequences having $n_i$ individuals of type $i$, as
\begin{align*}
	DS(\bo{n,r})=&\left\{ \bb{s}=(n_{i,j}(k),i,j\in[d],k\geq 0):\right.\\
	&\ \ \ \ \left. n_i=\sum_k n_{i,j}(k),\ n_j=r_j+\sum_i\sum_kk n_{i,j}(k), n_{i,j}(k)\geq 0\mbox{ for $i,j\in [d]$}\right\}.
\end{align*}Also, for any given multitype forest $\bo f$, define its empirical multitype degree sequence $\widehat{n}(\bo f):=\widehat{n}=(\widehat{n}_{i,j}(k),i,j\in [d],k\geq 0)$ as
\begin{equation*}
	\widehat{n}_{i,j}(k)=\sum_{\ell:v_\ell\in \bo f^{(i)}} \indi{\bo{c}_{i,j}(\ell)=k},
\end{equation*}where the sum is taken over all vertices of the subforest $\bo f^{(i)}$, and $\bo{c}_{i,j}(\ell)$ is the number of children type $j$, that the $\ell$-th individual of the subforest $\bo f^{(i)}$ of vertices type $i$ has.
Recall from Definition \ref{defiMFGDS} that $\mathbb{P}_{\bb s,\mathbf{r}}$ is the law of a MFGDS. 

\begin{lemma}\label{lemmacMBGWAreMixtureOfMFGDS}
Fix any $\mathbf{n}=(n_1,\ldots, n_d)\in \mathbb{N}^d$ and  $\mathbf{r}=(r_1,\ldots,r_d)\in \mathbb{Z}_+^d\setminus \{\bo 0\}$ such that  $\mathbf{r}\leq \mathbf{n}$. Let $\bb s=(n_{i,j}(k),i,j\in [d], k\geq 0) \in DS(\mathbf{n,r})$ be a multitype degree sequence. Consider a MBGW forest with offspring distribution $\bm{\mu}=(\mu_1,\ldots, \mu_d)$ satisfying the independence condition \hyperref[descriptionCMGIndependence]{\bo{H1}}. 
Then, the law of the MBGW forest conditioned on its multitype degree sequence being $\bb s$ is equal to $\mathbb{P}_{\bb s,\mathbf{r}}$. 
Furthermore, the law of a CMBGW($\mathbf{n},\bo r$) forest can be represented as a finite mixture of the laws $\{\mathbb{P}_{\bb s, \mathbf{r}} : \bb s\in DS(\mathbf{n,r})\}$.
\end{lemma}
\begin{proof}
	Let $\mathcal{F}_{\bo r}$ be a MBGW forest. 
	By the assumption on $\mu_i$, we can write $\mu_i(\bo{z})=\prod_j \mu_{i,j}(z_j)$ for any $\bo z=(z_1,\ldots, z_d)\in \z_+^d$, any $i$, and some laws $\mu_{i,j}$ on $\mathbb{Z}_+$. 
	Let $\bo f_1$ and $\bo f_2$ be two multitype forests having degree sequence $\bb{s}\in DS(\bo{n,r})$. 
	Then
	\begin{align*}
		\bo P\paren{\mathcal{F}=\bo f_1,\widehat{N}(\mathcal{F})=\bb{s} }
		& = \prod_{i\in [d]} \prod_{v\in  \bo f_1^{(i)}}\prod_{j\in [d]} \mu_{i,j}(\kappa_j(v))\\
		& = \prod_{i\in [d]} \prod_{j\in [d]} \prod_{k\geq 0} \mu_{i,j}(k)^{N_{i,j}(k)}\\
		& = \prod_{i\in [d]} \prod_{v\in \bo f_2^{(i)}}\prod_{j\in [d]} \mu_{i,j}(\kappa_j(v))\\
		&=\bo P\paren{\mathcal{F}=\bo f_2,\widehat{N}(\mathcal{F})=\bb{s} }.
	\end{align*}This implies the first assertion.
	To prove the second assertion, we sum over all the values in $DS(\bo{n,r})$, obtaining
	\begin{align*}
		\bo P\paren{\mathcal{F}\in \cdot|\bo{O}_{{\bf r}}=\bo{n}}
		& =\frac{1}{\bo P\paren{\bo{O}_{{\bf r}}=\bo{n}}}\sum_{\bb s\in DS(\bo{n,r})}\bo P\paren{\mathcal{F}\in \cdot,\widehat{N}(\mathcal{F})=\bb{s}}\\
		& =\sum_{\bb s\in DS(\bo{n,r})}\frac{\bo P\paren{\widehat{N}(\mathcal{F})=\bb{s},\bo{O}_{{\bf r}}=\bo{n}}}{\bo P\paren{\bo{O}_{{\bf r}}=\bo{n}}}\bo P\paren{\mathcal{F}\in \cdot|\widehat{N}(\mathcal{F})=\bb{s}}\\
		& =\sum_{\bb s\in DS(\bo{n,r})}\lambda_{\bb{s}}\p_{\bb s ,\bo r}\paren{\mathcal{F}\in \cdot},
	\end{align*}where
	\begin{equation*}
		\lambda_{\bb{s}}=\bo P\paren{\widehat{N}(\mathcal{F})=\bb{s}|\bo{O}_{{\bf r}}=\bo{n}}.
	\end{equation*}Note that trivially $\sum_{\bb s\in DS(\bo{n,r})}\lambda_\bb{s}=1$. 
\end{proof}

\subsection{Simulation of CMBGW($\bo{n},\bo r$) forests with given type sizes}

We consider here the simulation of MBGW forests conditioned to have individuals-type $\bo{n}$ and root-type $\bo{r}$. 
Using Devroye's idea 
we propose Algorithm \ref{algorithm8}. 
We denote by $\p(\cdot|\bo O_{\bo r}=\bo n)$ the law of a CMBGW($\bo{n},\bo r$) with root-type $\bo{r}$, and by $\mu_{i,j}$ the $j$th marginal of the distribution $\mu_i$.

We now prove Theorem \ref{teoSimulationOfCMBGWIntro}. 
At the end of the proof, we add an acceptance-rejection method to actually simulate the CMBGW($\bo{n},\bo r$) forest, which is the content of Algorithm \ref{algorithm8}.

\begin{proof}[Proof of Theorem \ref{teoSimulationOfCMBGWIntro}]
	Fix any $d$-type forest $\bo f$, having root-type $\bo r$, individuals-type $\bo n$, and degree sequence $\bb{s}=(n_{i,j}(k),i,j\in [d])$. 
	Using the same notation as in Theorem \ref{teoConstructionOFUniformMtypeForestWithGDS}, let $\bb w^b$ be a multidimensional bridge in $\mathbb{B}_{\bb s, \bo r}$, having multidimensional Vervaat transform $\bb w=V(\bb w^b,u)$ for some $u\in [\det (-\bb k)]$, where $k_{i,j}:=\sum k n_{i,j}(k)-n_i\indi{i=j}$. 
	
Let $(\bb{S}_{(j)};j\geq 1)$ be the degree sequence generated in the $j$-th repetition of the algorithm. 
Let $G$ be a geometric random variable with parameter $\p\paren{\Xi(\bo{S}_{i,j},i\in [d])=n_j,\ \forall\ j}$, and thus $\bb S=\bb S_{(G)}$, where $\bb S$ is the degree sequence described in the statement of the theorem.  
	Using that $\bb W^b$ has exchangeable increments, that $U$ is independent and uniform, and that there are $\prod n_i$ pairs $(\theta_{\bo{q,n}}(\bb w),u)$ that can be mapped to $\bb w$ (as seen on page \pageref{proofteoConstructionOFUniformMtypeForestWithGDS}), then
	\begin{align*}
	\p\paren{V(\bb{W}^b,U)=\bb w}& = \prod_i n_i\ \p\paren{\bb{W}^b=\bb w^b,U=u}\\
	& = \frac{1}{\frac{\det(-\bb k)}{\prod n_i}\prod_i \prod_j {n_i \choose n_{i,j}}}\p\paren{\bb S_{(K)}=\bb  s}\\
	& = \frac{1}{\frac{\det(-\bb k)}{\prod n_i}\prod_i \prod_j {n_i \choose n_{i,j}}}\frac{\p\paren{\bb S_{(1)}=\bb  s}}{\p\paren{\Xi(\bo{S}_{i,j},i\in [d])=n_j,\ \forall\ j}}.
	\end{align*}
	
	We compute explicitly the last fraction of the above equation. 
	For the term $\proba{\bb S_{(1)}=\bb s}$, we use the definition of the multinomial distribution
	\begin{equation*}
		\p\paren{\bb S_{(1)}=\bb s}=\prod_i \prod_j {n_i \choose n_{i,j}}\prod_{\ell\geq 0}\mu_{i,j}(\ell)^{n_{i,j}(\ell)}.
	\end{equation*}
    Notice that because the unconditioned random walk $\bb X$ has independent increments, the probability of generating the specific path $\bb w$ is exactly the product of the probabilities of its individual steps. Thus, $\p(\bb X = \bb w) = \prod_{i,j} \prod_{\ell \ge 0} \mu_{i,j}(\ell)^{n_{i,j}(\ell)}$. It follows that 
    \begin{equation*}
        \p\paren{\bb S_{(1)}=\bb s} = \p(\bb X = \bb w)\prod_i \prod_j {n_i \choose n_{i,j}}.
    \end{equation*}For the denominator we have
	\begin{align*}
	\p\paren{\Xi(\bo{S}_{i,j},i\in [d])=n_j,\ \forall\ j}& =\sum_{\substack{\bb{s}=(n_{i,j}):\\
			\sum_i\sum_k k n_{i,j}(k)=n_j-r_j,\forall\, j\\ \sum_k n_{i,j}(k)=n_i,\forall\ i }}\p\paren{\bb S_{(1)}=\bb s}\\
	& =\sum_{\substack{\bb{s}=(n_{i,j}):\\
			\sum_i\sum_k k n_{i,j}(k)=n_j-r_j,\forall\, j\\ \sum_k n_{i,j}(k)=n_i,\forall\ i }} \prod_i \prod_j {n_i \choose n_{i,j}}\prod_{\ell\geq 0}\mu_{i,j}(\ell)^{n_{i,j}(\ell)}.
	\end{align*}

	On the other hand, note that for fixed $j$, using the convolution formula:
	\begin{align*}
	\proba{\sum_{k=1}^d X^{k,j}_{n_k}=-r_j}& = \sum_{\sum_{k=1}^d\sum_{\ell=1}^{n_k} i_{k,\ell}=n_j-r_j}\prod_{k=1}^{d}\prod_{\ell=1}^{n_k}\mu_{k,j}(i_{k,\ell})\\
	& = \sum_{\substack{ \sum_i\sum_kk n_{i,j}(k)=n_j-r_j,\\\sum_k n_{i,j}(k)=n_i,\forall\ i} } \prod_i {n_i \choose n_{i,j}}\prod_{\ell\geq 0}\mu_{i,j}(\ell)^{n_{i,j}(\ell)},
	\end{align*}where in the last equality, we used the fact that $\prod_i{n_i \choose n_{i,j}}$ is the number of different bridges having the same degree sequence $(n_{1,j}(0), n_{1,j}(1),\ldots ), \ldots,(n_{d,j}(0), n_{d,j}(1),\ldots )$. 
	Thus, multiplying for all $j$ we have
	\begin{equation}\label{eqnXiOfSAndProductOfDensities}
		\p\paren{\Xi(\bo{S}_{i,j},i\in [d])=n_j,\ \forall\ j}=\prod_j \proba{\sum_k X^{k,j}_{n_k}=-r_j}. 
	\end{equation}
	
Substituting these into our initial equation, we obtain:
	\begin{align*}
	\p\paren{V(\bb{W}^b,U)=\bb w}& =  \frac{1}{\frac{\det(-\bb k)}{\prod n_i}}\frac{ \p(\bb X = \bb w) }{\prod_j \proba{\sum_k X^{k,j}_{n_k}=-r_j}}\\ 
	& =  \frac{1}{\frac{\det(-\bb k)}{\prod n_i}}\frac{\p\paren{\bb{X}=\bb w}}{\frac{\prod n_i}{\det\paren{-\widehat{\bb k}(\bo{r}, \bo{n})}}\p\paren{\bo T_{\bo r}=\bo n}}\\
	& =  \frac{\det\paren{-\widehat{\bb k}(\bo{r}, \bo{n})}}{\det(-\bb k)}\p\paren{\bb{X}=\bb w|\bo T_{\bo r}=\bo n}, 
	\end{align*} 
    where the second equality follows from \eqref{eqnGeneralDeterminantFormula} in Theorem \ref{teoLawOfTheTotalPopulationByTypes}.
This proves the first part of the theorem.

Note that under Hypothesis \hyperref[descriptionCMGConvolution]{\bo{H2}}, by \eqref{eqnOterDwassGeneralizationH2} we obtain 
	\begin{align*}
	\p\paren{V(\bb{W}^b,U)=\bb w} =\frac{1}{\frac{\# \bo n}{\# \bo r}\frac{\det(-\bb k)}{\prod n_i}}\p\paren{\bb{X}=\bb w|\bo T_{\bo r}=\bo n}.
	\end{align*}

From Algorithm \ref{algorithm8}, the first 9 steps constitute a standard acceptance-rejection method (see \cite[Section 8.2.4]{MR0630193}) applied directly to the degree sequence. 
	For each multitype forest coded by $\bb w$, with root-type $\bo{r}$ and individuals-type $\bo{n}$, define the likelihood ratio between the target conditional density $f(\bb w) := \p\paren{\bb X=\bb w \mid \bo T_{\bo r}=\bo n}$ and the proposal density $g(\bb w) := \p\paren{V(\bb W^{b,\bb S},U)=\bb w}$ as:
	\begin{equation*}
		c_{\bb w} = \frac{f(\bb w)}{g(\bb w)} = \frac{\det(-\bb k)}{\det\paren{-\widehat{\bb k}(\bo r, \bo n)}},
	\end{equation*}
	(note that $\bb k$ depends on $\bb w$ only through its terminal values $k_{i,j}=\sum k n_{i,j}(k)-n_i\indi{i=j}$). 
	By Lemma \ref{lemmaNumberOfGoodCyclicalPerm},, $\det(-\bb k)$ counts the exact number of good cyclical permutations of the bridge $\bb w^b$. Since the total number of cyclical permutations available in the multidimensional grid is exactly $\prod_{i=1}^d n_i$, we have $\det(-\bb k) \le \prod_{i=1}^d n_i$. 
	
	This provides us with a universal majorizing constant $c$ given by $
        c_{\bb w} \leq \det\paren{-\widehat{\bb k}(\bo r, \bo n)}^{-1}\prod_{i=1}^d n_i =: c.
    $
	We can thus apply the acceptance-rejection method to simulate a CMBGW($\bo n,\bo r$) forest, accepting the generated breadth-first walk $\bb W$ drawn from $g(\cdot)$ when an independent uniform variable $V$ on $[0,1]$ satisfies:
	\begin{equation*}
		V \leq \frac{f(\bb W)}{c \cdot g(\bb W)} = \frac{c_{\bb W}}{c} = \frac{\det(-\bb k) / \det(-\widehat{\bb k}(\bo r, \bo n))}{\prod_{i=1}^d n_i / \det(-\widehat{\bb k}(\bo r, \bo n))} = \frac{\det(-\bb k)}{\prod_{i=1}^d n_i}\leq 1.
	\end{equation*}
	Notice that the unknown theoretical determinant $\det\paren{-\widehat{\bb k}(\bo r, \bo n)}$ cancels out of the acceptance ratio. 
This concludes the proof.
\end{proof}

\section{Expected complexity}\label{sectionExpectedComplexity}

In this section we prove that our algorithm has a significantly smaller expected complexity than a standard na\"ive algorithm to produce a CMBGW forest.
In order to do so, we need a tight control of the terms $\mu^{*\bo n}_j(-r_j)$ and $\proba{\bo T_{\bo r} = \bo n}$ as $\bo n$ goes to infinity. 
Thus, first we establish local limit theorems for both probabilities. 
The proof follows easily once we established Theorem \ref{teoLawOfTheTotalPopulationByTypes}. 

\begin{proof}[Proof of Lemma \ref{lemmaLLTTargetSums}]
	For a fixed $j \in [d]$, the term $\sum_{\ell=1}^d X^{\ell,j}_{n_\ell}$ is a sum of $\# \bo n$ independent, integer-valued random variables. Since there are only $d$ distinct distributions and $n_\ell \to \infty$ proportionally to $\# \bo n$, and the ${\rm g.c.d.}(H_\ell;\ell\in D)=1$, then \cite{MR388499} Chapter VII, Theorem 2 implies the first result. 
	The second convergence regarding $\proba{\bo T_{\bo r} = \bo n}$ follows trivially from the first convergence and Theorem \ref{teoLawOfTheTotalPopulationByTypes}. 

	Finally, since $\sigma^{(\bo n)}_j = \Theta(\# \bo n)$ and $|-r_j - m^{(\bo n)}_j| \le c \sqrt{\# \bo n}$, the exponential term is bounded strictly away from $0$, yielding the last stated asymptotic.
\end{proof}

We now provide an explicit construction for the sizes $\bo n $ that strictly satisfies the necessary condition to apply the local limit theorem without incurring an exponential decay. 

\begin{example}[Example of permissible sizes by types]\label{lemmaExplicitConstruction}
Assume the multitype Bienaym\'e-Galton-Watson process is irreducible and critical, so the spectral radius of the mean matrix $\bb M = (m_{\ell,j})_{\ell,j=1}^d$ is $1$. Let $\bm{\xi} = (\xi_1, \dots, \xi_d)$ be the strictly positive left eigenvector of $\bb M$ corresponding to the eigenvalue $1$, normalized such that $\sum_{i=1}^d \xi_i = 1$. 

For any integer $N \in \mathbb{N}$ and any given functions $f_i(N) = O(\sqrt{N})$ for $i \in [d]$, define the type sizes as $n_i := \floor{N \xi_i + f_i(N)}, \quad i \in [d]$.
Then $\# \bo n = \Theta(N)$, and for any fixed root-type $\bo r \in \mathbb{Z}^d_+$, the deviation satisfies $|-r_j - m^{(\bo n)}_j| = O(\sqrt{\# \bo n})$. 
\end{example}
\begin{proof}
By definition, $n_\ell = N \xi_\ell + f_\ell(N) - \varepsilon_\ell$ for some fractional part $\varepsilon_\ell \in [0, 1)$. Substituting this into the definition of $m^{(\bo n)}_j$:
\begin{align*}
	m^{(\bo n)}_j &= \sum_{\ell=1}^d (N \xi_\ell + f_\ell(N) - \varepsilon_\ell) m_{\ell,j} - (N \xi_j + f_j(N) - \varepsilon_j) \\
	&= N \paren{\sum_{\ell=1}^d \xi_\ell m_{\ell,j} - \xi_j} + \sum_{\ell=1}^d f_\ell(N) m_{\ell,j} - f_j(N) - \sum_{\ell=1}^d \varepsilon_\ell m_{\ell,j} + \varepsilon_j.
\end{align*}
Since $\bm{\xi}$ is the left eigenvector corresponding to the critical eigenvalue $1$, we have $\sum_{\ell=1}^d \xi_\ell m_{\ell,j} = \xi_j$. Therefore, the first term cancels. 
Since $f_i(N) = O(\sqrt{N})$ and $\varepsilon_i = O(1)$, it follows that $m^{(\bo n)}_j = O(\sqrt{N})$. 
Because $r_j$ is a fixed constant and $\# \bo n = \Theta(N)$, we deduce the result. 
\end{proof}

\subsection{Expected time complexity of the na\"ive method}\label{subsectionComplexityNaive}

The idea to construct a forest with a given size using the na\"ive method is an easy extension of the unitype case. To construct a unitype forest with $r$ roots of size $n$ having offspring distribution $\mu$, consider random walks $(X(\ell), \ell \ge 1)$ with step distribution $\mu(\cdot + 1)$. Now, keep generating copies of random walks of the form $(X(\ell), \ell \ge 1)$ until the first time it hits $-r$ exactly at time $n$, that is $\min\{\ell : X(\ell) = -r\} = n$. When that happens, such an excursion codes the desired forest of size $n$.

In the multitype case, we use random variables $(X^{i,j}(\ell), \ell \ge 1)$ with step distribution $\mu_i(\cdot + \bo e_i)$ for vertices of type $i$, and repeat until we obtain a multidimensional first passage bridge of length $\bo n$. Recall from Equation \eqref{eqnDefinitionT_rx} that such first time ensures we obtain a forest. 

The na\"ive algorithm is the following.

\begin{algorithm}[H]
	\caption{Na\"ive algorithm to generate a CMBGW forest}
	\label{algorithmGenerateNaiveCMGW}
	\begin{algorithmic}[1]
		\STATEx {\bf Input: }A critical offspring distribution $\bm \mu=(\mu_1,\ldots, \mu_d)$, and vectors $\bo r \in \mathbb{Z}^d_+ \setminus \{\bo 0\}$, $\bo n \in \mathbb{N}^d$.
		\STATEx {\bf Output: }A CMBGW($\bo n,\bo r$) forest.
		\STATE For every $i\in [d]$, generate the vectors $(X^{i,j}(k),j\in [d],k\in [n_i])$ using $\bm \mu$. 
		\STATE Construct the terminal sums $\bb X_{\bo n}$. 
		\IF{$r_j + (\bo 1 \cdot \bb X_{\bo n})_j \neq 0$ for some $j \in [d]$}
		    \STATE \textbf{Reject} (terminal state does not match) and repeat from Step 1.
		\ENDIF
		\IF{there exists $\bo m < \bo n$ such that $r_j + (\bo 1 \cdot \bb X_{\bo m})_j = 0$ for all $j \in [d]$}
		    \STATE \textbf{Reject} (not a minimal solution) and repeat from Step 1.
		\ENDIF
		\STATE \textbf{Accept} and construct the forest coded by $\bb X$.
	\end{algorithmic}
\end{algorithm}

Let us compute the expected time to obtain a CMBGW forest using this na\"ive method. The total expected time is the expected cost of a single iteration multiplied by the expected number of iterations. 
We assume the hypotheses of Lemma \ref{lemmaLLTTargetSums} are satisfied. 

Let $A^{(\bo n)}$ be the event that the generated path reaches the exact target sum at the end, i.e., $A^{(\bo n)} = \{ \sum_{\ell=1}^d X^{\ell,j}_{n_\ell} = -r_j \text{ for all } j \in [d] \}$. This means the walk forms a valid multidimensional bridge.
Let $B^{(\bo n)}$ be the event that the path actually codes a valid forest, meaning $\bo n$ is a multidimensional first passage bridge. Note that $\{\bo T_{\bo r} = \bo n\}=B^{(\bo n)} \subset A^{(\bo n)}$.

In a single independent iteration of Algorithm \ref{algorithmGenerateNaiveCMGW}, the computational costs are:
\begin{enumerate}
    \item Generating the full walk takes $\Theta(\#\bo n)$ time.
    \item Checking if the terminal state matches (event $A^{(\bo n)}$) takes $\Theta(d)$ time.
    \item If event $A^{(\bo n)}$ occurs, we perform the minimality check to see if $B^{(\bo n)}$ occurs. 
Since $n_i = \Theta(\#\bo n)$, checking all $\bo m \le \bo n$ na\"ively takes $\Theta\paren{\prod_{i=1}^d n_i}$ time.
\end{enumerate}

Therefore, the expected time spent in a \emph{single iteration} is:
\begin{equation*}
    \esp{\text{Time per iteration}} = \Theta(\#\bo n) + \proba{A^{(\bo n)}} \Theta\paren{\prod_{i=1}^d n_i}.
\end{equation*}

Because the iterations are independent and we reject until $B^{(\bo n)}$ occurs, the expected number of attempts is $\esp{I} = 1/\proba{B^{(\bo n)}}$.
By Wald's equation, the total expected time complexity of the na\"ive method is:
\begin{align*}
    \esp{T_{\text{na\"ive}}} &= \frac{1}{\proba{B^{(\bo n)}}} \esp{\text{Time per iteration}} \\
    &= \frac{1}{\proba{B^{(\bo n)}}} \Theta(\#\bo n) + \frac{\proba{A^{(\bo n)}}}{\proba{B^{(\bo n)}}} \Theta\paren{\prod_{i=1}^d n_i}.
\end{align*}

By Theorem \ref{teoLawOfTheTotalPopulationByTypes}, assuming Hypothesis \hyperref[descriptionCMGConvolution]{\bo{H2}} holds, we have exactly:
\begin{equation*}
    \proba{B^{(\bo n)}} =  \frac{\# \bo r}{\# \bo n} \proba{A^{(\bo n)}}.
\end{equation*}
Applying Lemma \ref{lemmaLLTTargetSums}, the expected total time reduces to:
\begin{equation*}
    \esp{T_{\text{na\"ive}}} = \Theta\paren{ (\# \bo n)^{d/2 + 2} +  (\#\bo n)^{d+1}} = \Theta\paren{ (\# \bo n)^{\max(d/2 + 2, \, d+1)}}.
\end{equation*}

Notice that for any $d \ge 2$, the minimality verification formally dominates, requiring $\Theta((\#\bo n)^{d+1})$ time. This extreme polynomial bottleneck clearly highlights the necessity of our proposed approach.

\subsection{Expected time complexity of Algorithm \ref{algorithm8}}\label{subsectionExpTimeComplexity}

The computations in this section generalize the expected time analysis from \cite[Section 5]{MR2888318}. By evaluating the expected time of a single iteration and applying Wald's equation, we highlight the efficiency of combining the multidimensional Vervaat transform with the early degree-sequence rejection natively executed in Algorithm \ref{algorithm8}.

In Algorithm \ref{algorithm8}, let $A^{(\bo n)}$ be the event that the generated multinomials satisfy the sum condition $\Xi_j = n_j$ for all $j \in [d]$ (Step 4). Because the row sums correspond to the independent steps of the random walks, the probability of this event is exactly $\proba{A^{(\bo n)}} = \prod_{j=1}^d \mu^{*\bo n}_j(-r_j)$. 
Let $B^{(\bo n)}$ be the event that the degree sequence is subsequently accepted via Step 9. 
Note that $B^{(\bo n)} \subset A^{(\bo n)}$.
    
In a single independent iteration of the rejection loop (Steps 1 to 11), the computational costs are:
\begin{enumerate}
    \item Generating the multinomial vectors (Steps 1--3). Denoting by $\tau_{i,j}(n_i)=\esp{K_{i,j}}$ the expected maximum number of children, this takes expected time $\Theta\paren{d^2 + \sum_{i,j} \tau_{i,j}(n_i)}$.
    \item Checking if the sums match (event $A^{(\bo n)}$ in Step 4) takes $\Theta(d)$ time.
    \item If event $A^{(\bo n)}$ occurs, we perform the acceptance check (Steps 7--11). Computing the determinant $\det(-\bb k)$ takes $\Theta(d^3)$ time.
\end{enumerate}
Here we generate multinomial vectors sequentially using binomial random variables. Each binomial can be simulated in expected time $O(1)$, as explained in \cite{MR2888318}. 
Therefore, the expected time spent in a \emph{single iteration} of the rejection loop is:
\begin{equation*}
    \esp{\text{Time per iteration}} = \Theta\paren{d^2 + \sum_{i,j} \tau_{i,j}(n_i)} + \proba{A^{(\bo n)}} \cdot \Theta(d^3).
\end{equation*}

Because the iterations are independent and we reject and repeat from Step 1 until $B^{(\bo n)}$ occurs, the expected number of attempts is $\esp{I} = 1/\proba{B^{(\bo n)}}$. By Wald's equation, the total expected time spent in the rejection loop is:
\begin{equation*}
\begin{split}
    \esp{T_{\text{loop}}} &= \frac{1}{\proba{B^{(\bo n)}}} \esp{\text{Time per iteration}} \\
    &= \frac{1}{\proba{B^{(\bo n)}}} \Theta\paren{d^2 + \sum_{i,j} \tau_{i,j}(n_i)} + \frac{\proba{A^{(\bo n)}}}{\proba{B^{(\bo n)}}} \Theta(d^3).
\end{split}
\end{equation*}

Let us compute the exact probability of $B^{(\bo n)}$. Consider any fixed multitype degree sequence $\bb s\in DS(\bo n, \bo r)$, where recall that $DS(\bo n, \bo r)$ is the set of all multitype degree sequences satisfying the terminal sum conditions of $A^{(\bo n)}$.
Note that $A^{(\bo n)}\cap \{\bb S = \bb s \}=\{\bb S = \bb s\}$. 
Observe that $B^{(\bo n)}$ occurs if $A^{(\bo n)}$ does and the degree sequence $\bb S$ passes the randomized acceptance check in Steps 9--11. 
Thus, the conditional probability of acceptance given $A^{(\bo n)}$ and a specific degree sequence $\bb S = \bb s$ is exactly $\det(-\bb k(\bb s)) / \prod n_i$. 
Here we put the dependency of $\bb k$ on the degree sequence $\bb s$. 
Thus, the unconditional probability is
\begin{equation*}
    \proba{B^{(\bo n)}} = \esp{ \bo 1\big\{A^{(\bo n)}\big\} \frac{\det(-\bb k(\bb S))}{\prod n_i} } = \sum_{\bb s \in DS(\bo n, \bo r)} \proba{\bb S = \bb s} \frac{\det(-\bb k(\bb s))}{\prod n_i}.
\end{equation*}

Recall that $\mathbb{B}_{\bb s, \bo r}$ is the set of all multidimensional paths of length $\bo n$ having exactly this degree sequence $\bb s$. 
Recall also that the size of this set is given by the product of multinomial coefficients $\# \mathbb{B}_{\bb s, \bo r}= \prod_{i,j} \binom{n_i}{n_{i,j}(0), n_{i,j}(1), \ldots}$. For any path $\bb x \in \mathbb{B}_{\bb s, \bo r}$, the probability that the unconditioned random walks $\bb X$ exactly match $\bb x$ is $\p(\bb X = \bb x) = \prod_{i,j} \prod_k \mu_{i,j}(k)^{n_{i,j}(k)}$. Hence, the multinomial probability of generating $\bb s$ can be identically written as the sum of its corresponding path probabilities $\proba{\bb S = \bb s} = \sum_{\bb x \in \mathbb{B}_{\bb s, \bo r}} \p(\bb X = \bb x)$. 
Substituting this back into the expectation yields:
\begin{equation*}
    \proba{B^{(\bo n)}} = \sum_{\bb s \in DS(\bo n, \bo r)} \sum_{\bb x \in \mathbb{B}_{\bb s, \bo r}} \p(\bb X = \bb x) \frac{\det(-\bb k(\bb x))}{\prod n_i} = \sum_{\bb x \in \mathbb{B}_{\bo n, \bo r}} \p(\bb X = \bb x) \frac{\det(-\bb k(\bb x))}{\prod n_i},
\end{equation*}
where $\mathbb{B}_{\bo n, \bo r}$ is the complete set of all multidimensional paths of length $\bo n$ matching the terminal sum conditions imposed by $A^{(\bo n)}$ (i.e., multidimensional bridges).
By Corollary \ref{coroRelatingFirstHittingTimeAndBridges}, the right-hand side is exactly $\p(\bo T_{\bo r} = \bo n)$.

By Theorem \ref{teoLawOfTheTotalPopulationByTypes} (Equation \ref{eqnOterDwassGeneralizationH2}), under Hypothesis \hyperref[descriptionCMGConvolution]{\bo{H2}}, we have exactly $\proba{B^{(\bo n)}} = 
\frac{\# \bo r}{\# \bo n} \proba{A^{(\bo n)}}$.

Substituting this ratio back into the expected time equation gives
\begin{equation*}
    \esp{T_{\text{loop}}} = \Theta\paren{ \frac{\# \bo n}{\# \bo r} \paren{ \frac{d^2 + \sum_{i,j}\tau_{i,j}(n_i)}{\prod_{j=1}^d \mu^{*\bo n}_j(-r_j)} + d^3} }.
\end{equation*}

Once event $B^{(\bo n)}$ occurs, the rejection loop terminates and Algorithm \ref{algorithmUMFGDS} is executed (Step 12) on the single accepted degree sequence. The computational costs for this final phase are:
\begin{enumerate}
    \item The child sequence and the $d^2$ uniform permutations $(\bm \pi_{i,j})$ are generated in $\Theta(d\sum n_i) = \Theta(\# \bo n)$ expected time.
    \item To compute the multidimensional Vervaat transform, we must find a good cyclical permutation among the $\det(-\bb k)$ valid ones. Because $X^{i,i}$ is the BFW of the subforest of type $i$, we know $n_i=\min\{\ell: X^{i,i}_\ell=\min_{0\leq k\leq n_i}X^{i,i}_k \}$. Thus, we only need to consider joint cyclical permutations that shift each coordinate to one of its first times that reach a new minimum. There are exactly $X^{i,i}_{n_i}$ such minimum for each type $i$. The number of candidate joint permutations is exactly $\prod_{i=1}^d (-X^{i,i}_{n_i})$. We test candidates from this restricted pool to find a valid forest (which takes $\Theta(\#\bo n)$ time per test, as proved in Lemma \ref{lemmaExpectedTimeToShowItIsAValidForest} below). The expected time for the Vervaat search is thus bounded by $\Theta\paren{ \#\bo n \, \E\bra{ \prod_{i=1}^d (-X^{i,i}_{n_i}) \;\middle|\; B^{(\bo n)} } }$.
\end{enumerate}

Summing the expected time of the rejection loop and the single execution of Algorithm \ref{algorithmUMFGDS} directly yields the upper bound for the overall complexity stated in Theorem \ref{teoIntroComplexityTime}.

To understand the asymptotic behavior, consider the case where the offspring distribution has finite variance. The condition $\sum_{\ell=1}^d X_{n_\ell}^{\ell,j} = -r_j$ forces $ X_{n_i}^{i,i}$ to scale on the order of $\Theta(\sqrt{n_i})$. Therefore, the conditional expectation of the product $\prod_{i=1}^d (-X^{i,i}_{n_i})$ scales as $\Theta((\# \bo n)^{d/2})$, making the Vervaat step take expected $\Theta((\# \bo n)^{d/2+1})$ time.
Furthermore, the maximum step size $\tau_{i,j}(n_i)$ scales as $o(\sqrt{n_i})$. The probability of $A^{(\bo n)}$ is $\Theta((\# \bo n)^{-d/2})$ by Lemma \ref{lemmaLLTTargetSums}. Thus, the expected time spent in the rejection loop is:
\begin{equation*}
    o\paren{ \frac{\# \bo n}{\# \bo r} \cdot \frac{\sqrt{\# \bo n}}{(\# \bo n)^{-d/2}} } = o\paren{ \frac{(\# \bo n)^{(d+3)/2}}{\# \bo r} }.
\end{equation*}

Since $\frac{d+3}{2} > \frac{d}{2} + 1$ for any $d \ge 1$, the generation of the multinomial degree sequences asymptotically dominates the runtime, bringing the total expected complexity to $o\paren{\frac{(\# \bo n)^{(d+3)/2}}{\# \bo r}}$.

This complexity highlights the great advantage of our method. A na\"ive approach (as calculated in Section \ref{subsectionComplexityNaive}) requires an expected time of $\Theta((\# \bo n)^{\max(d/2+2, \, d+1)}/\# \bo r)$ due to the exhaustive minimality checks on every valid multidimensional bridge. By testing the determinant condition \emph{before} applying the Vervaat transform and restricting the search space to the strict minima of the diagonal walks, Algorithm \ref{algorithm8} bypasses this severe bottleneck.

\section*{Appendix}

We prove here Lemma \ref{lemmaPairsYUMappedToW}, which is used in the proof of Theorem \ref{teoConstructionOFUniformMtypeForestWithGDS} to count the preimages of the multidimensional Vervaat transform.

Let $s, m \in \mathbb{N}$. Define the set of paths finishing at $-m$ at time $s$ as $\mathbb{B}_{s,m} = \left\{ y \in \mathbb{Z}^s : y(j)-y(j-1) \ge -1, y(s)=-m \right\}$.
For $i \in [s-1]_0$, the cyclic permutation $\theta_i(y)$ shifts the increments of $y$ cyclically by $i$ steps. For $u \in [m-1]_0$, let $\tau_u(y)$ be the first time $y$ hits $\min(y) + u$. The Vervaat-type transformation is defined as $V(y,u) = \theta_{\tau_u(y)}(y)$, which maps $\mathbb{B}_{s,m}$ to the set of excursions $\mathbb{E}_{s,m}$ (paths in $\mathbb{B}_{s,m}$ that hit $-m$ for the first time exactly at step $s$).
They are also called first-passage bridges. 

Given a degree sequence $\bo s = (n_i, i \ge 0)$ with $\#\bo s: = \sum n_i$ and $\sum i n_i = \#\bo s - m$, let $\mathbb{B}_{\bo s}$ and $\mathbb{E}_{\bo s}$ denote the restrictions of $\mathbb{B}_{\#\bo s,m}$ and $\mathbb{E}_{\#\bo s,m}$ to paths whose increment counts exactly match $\bo s$.

The following lemma is similar to a classical combinatorial lemma about ladder indices and cyclical permutations (cf.\ \cite[XII.6, Lemma 1]{MR0270403}), and is inspired by Lemma 2.2 of \cite{MR4003553}.

\begin{lemma}\label{lemmaPairsYUMappedToW}
    Let $w\in \mathbb{E}_{\bo{s}}$. Then, the number of different pairs $(y, u)\in \mathbb{B}_{\bo{s}}\times [m-1]_0$ such that $V(y,u)=w$ is exactly $\#\bo s$. 
\end{lemma}
\begin{proof}
    If $V(y, u) = w$, then $y$ must be a cyclic permutation of $w$. Thus, any such $y$ is of the form $y_i = \theta_i(w)$ for some $i \in [\#\bo s-1]_0$. We must find for which $u \in [m-1]_0$ we have $V(y_i, u) = w$.
    
    Since $w \in \mathbb{E}_{\bo S}$, $w$ reaches $-m$ for the first time at step $\#\bo s$. For $y_i = \theta_i(w)$, its values are given by $y_i(j) = w(j+i) - w(i)$ for $j \le \#\bo s-i$. Thus $y_i(\#\bo s-i) = w(\#\bo s) - w(i) = -m - w(i)$. 
    Because $w(j) > -m$ for all $j < \#\bo s$, $y_i$ cannot reach $-m - w(i)$ strictly before step $\#\bo s-i$.
    The overall minimum of $y_i$ over the full path of length $\#\bo s$ is given by $\min(y_i) = \min_{k \le i} w(k) - m - w(i)$.
    
    In order for the Vervaat transform to shift $y_i$ exactly at time $\#\bo ss-i$ to recover $w$, we need $\#\bo ss-i$ to be the \emph{first} time $y_i$ hits $\min(y_i) + u$. 
    Therefore, we must set $u = y_i(\#\bo s-i) - \min(y_i) = (-m - w(i)) - (\min_{k \le i} w(k) - m - w(i)) = -\min_{k \le i} w(k)$.
    Because $w \in \mathbb{E}_{\bo s}$, its running minimum before step $s$ is between $-m+1$ and $0$, ensuring $u \in [0, m-1]$.
    
    This shows that for each $i \in [\#\bo ss-1]_0$, the pair $(y_i, u_i)$ with $y_i = \theta_i(w)$ and $u_i = -\min_{k \le i} w(k)$ maps to $w$. 
    Let us prove that all $\#\bo s$ pairs are distinct. If $(y_i, u_i) = (y_j, u_j)$ for $i < j$, then $w$ must be periodic, meaning $w(k + l) = w(k) + w(l)$ for $l = j-i$. But because $w(\#\bo s) = -m < 0$, $w(l)$ must be strictly negative. Consequently, the running minimum $\min_{k \le j} w(k)$ will be strictly smaller than $\min_{k \le i} w(k)$, implying $u_i \ne u_j$. Thus, all $\#\bo s$ pairs are distinct, and the number of preimages is exactly $\#\bo s$.
\end{proof}

The following result is needed to finish the Proof of Theorem \ref{teoLawOfTheTotalPopulationByTypes}. 

\begin{lemma}\label{lemmaDeterminantUnderHypothesis}
	\[
	\frac{1}{\prod n_i}\sum_{\mathfrak{e}\in \bb F_d^{\text{elem.}}}\prod_{j=1}^{d}\widetilde{k}_{\mathfrak{e}_j,j}=\frac{r_1+\cdots +r_d}{n_1+\cdots +n_d}.
	\]
\end{lemma}
\begin{proof}
Recall that $\#\bo n=\sum n_i$ and $\#\bo r=\sum r_i$. 
Define the matrix $\overline{\bb k}$ as a $d\times d$ matrix with entries $\overline{k}_{i,j}=\#\bo n\widetilde{k}_{i,j}=n_i(n_j-r_j)$ for $i\neq j$, and diagonal
	\begin{equation*}
		-\overline{k}_{j,j}=\#\bo nr_j+\sum_{i\neq j}\#\bo n\widetilde{k}_{i,j}=\#\bo nr_j+\sum_{i\neq j}n_i(n_j-r_j)=\#\bo nr_j+(n_j-r_j)(\#\bo n-n_j)=n_j(\#\bo n-n_j+r_j).
	\end{equation*}Then, using Lemma \ref{lemmaNumberOfGoodCyclicalPerm}, we have
	\begin{equation*}
		\sum_{\mathfrak{e}\in \bb F_d^{\text{elem.}}}\prod_{j=1}^{d}\widetilde{k}_{\mathfrak{e}_j,j}=\#\bo n^{-d}det(-\overline{\bb k}).
	\end{equation*}To prove that $det(-\overline{\bb k})=\#\bo n^d\frac{\#\bo r}{\#\bo n}\prod n_i $, factorize in row $i$ the factor $n_i$, obtaining
	\begin{equation*}
		det(-\overline{\bb k}) = \prod n_i
		\begin{vmatrix}
			\#\bo n-n_1+r_1& -(n_2-r_2) & \cdots & -(n_d-r_d) \\
			\vdots  & \vdots  & \ddots & \vdots  \\
			-(n_1-r_1) & -(n_2-r_2) & \cdots & 	\#\bo n-n_d+r_d
		\end{vmatrix}.
	\end{equation*}Multiply the last row by minus one, and add it to every other row, to obtain
	\begin{equation*}
		det(-\overline{\bb k}) = \prod n_i
		\begin{vmatrix}
			\#\bo n& 0 & \cdots & -\#\bo n \\
			0& \#\bo n & \cdots & -\#\bo n \\
			\vdots  & \vdots  & \ddots & \vdots  \\
			-(n_1-r_1) & -(n_2-r_2) & \cdots & 	\#\bo n-n_d+r_d
		\end{vmatrix}.
	\end{equation*}Multiply by $(n_i-r_i)/\#\bo n$ each row $i\in [d-1]$, and add it to the last row
	\begin{equation*}
		det(-\overline{\bb k}) = \prod n_i
		\begin{vmatrix}
			\#\bo n& 0 & \cdots & -\#\bo n \\
			0& \#\bo n & \cdots & -\#\bo n \\
			\vdots  & \vdots  & \ddots & \vdots  \\
			0 & 0 & \cdots & 	\#\bo n-\sum_1^d(n_i-r_i)
		\end{vmatrix},
	\end{equation*}and, being an upper triangular, it follows that $det(-\overline{\bb k}) = \#\bo n^{d-1}\#\bo r\prod n_i=\#\bo n^d\frac{\#\bo r}{\#\bo n}\prod n_i$ as wanted.
\end{proof}

\begin{lemma}\label{lemmaExpectedTimeToShowItIsAValidForest}
Given a bridge $\bb x$ in $S_d$ with length $\bo n$, the time to validate if it codes a forest is $\Theta(\#\bo n)$.
\end{lemma}
\begin{proof}
To verify whether a candidate cyclical permutation $\bb x$ properly codes a valid multitype forest, we do not need to na\"ively evaluate the multidimensional minimal solution condition $\bo r + \bo 1 \cdot \bb x_{\bo m} > \bo 0$ across all $\prod_{i=1}^d n_i$ grid points $\bo m < \bo n$. Instead, we can structurally validate the candidate natively in $\Theta(\# \bo n)$ time by sequentially simulating the breadth-first exploration process of the subforests (cf. Section 2.2 of \cite{MR3449255}). 

The validation algorithm proceeds sequentially as follows:
\begin{enumerate}
    \item Initialize $d$ active counters (representing the queues of 'Active' unexplored vertices of each type) with the initial roots: $\bo a=\bo r$. Initialize the number of 'Explored' vertices $\bo e=\bo 0$.
    \item Loop at most $\#\bo n = \sum_{i=1}^d n_i$ times:
    \begin{itemize}
        \item Find type $i^*:=\min\{i \in [d]: e_i<n_i,a_i>0\}$. If no such type exists, but there are still unprocessed vertices overall (i.e., $\sum_{i=1}^d e_i < \#\bo n$), a subforest died prematurely. In this case, we stop and \textbf{reject} the candidate.
        \item Pick the active vertex, say $v_{i^*}$, of type $i^*$ with smallest label. 
Update the status of this vertex, moving if from active to explored: increment $e_{i^*} \gets e_{i^*} + 1$, and decrement $a_{i^*} \gets a_{i^*} - 1$. 
Add all children of $v_{i^*}$ to the active vertices: increment $\bo a \gets \bo a+(\kappa_1(v_{i^*}),\ldots, \kappa_d(v_{i^*}))$. 
    \end{itemize}
    \item The candidate path is a valid multidimensional breadth-first walk if and only if the exploration completes exactly $\#\bo n$ steps (i.e., it terminates exactly when $\bo e=\bo n$).
\end{enumerate}
\end{proof}

\bibliography{OsBib}
\bibliographystyle{amsalpha}
\end{document}